\documentclass[a4paper]{amsart}

\usepackage[utf8]{inputenc}

\usepackage[american]{babel}
\usepackage{amsthm}
\usepackage{amssymb}
\usepackage{amsmath}
\usepackage[shortlabels]{enumitem}
\usepackage{mathrsfs}
\usepackage{mathtools}
\usepackage{url}

\usepackage{tikz} 
\usetikzlibrary{matrix,arrows,calc,decorations.pathreplacing,fit}
\usepackage{tikz-cd}

\numberwithin{equation}{section}

\newtheorem{theorem}{Theorem}[section]
\newtheorem{lemma}[theorem]{Lemma}
\newtheorem{corollary}[theorem]{Corollary}
\newtheorem{proposition}[theorem]{Proposition}

\theoremstyle{definition}
\newtheorem{definition}[theorem]{Definition}

\newtheorem{notation}[theorem]{Notation}
\theoremstyle{remark}
\newtheorem{remark}[theorem]{Remark}
\newtheorem{remark*}[theorem]{Remark}
\newtheorem{example}[theorem]{Example}

\DeclareMathOperator{\Aut}{Aut}

\DeclareMathOperator{\End}{End}
\DeclareMathOperator{\Fil}{Fil}

\DeclareMathOperator{\id}{id}
\DeclareMathOperator{\image}{im}

\DeclareMathOperator{\Gal}{Gal}

\DeclareMathOperator{\GSp}{GSp}

\DeclareMathOperator{\Hom}{Hom}
\DeclareMathOperator{\Irr}{Irr}
\DeclareMathOperator{\Isom}{Isom}

\DeclareMathOperator{\pr}{pr}
\DeclareMathOperator{\QIsog}{QIsog}

\DeclareMathOperator{\Sh}{Sh}

\DeclareMathOperator{\Spec}{Spec}
\DeclareMathOperator{\Spf}{Spf}

\def \AA {\mathbb{A}}

\def \DD {\mathbb{D}}

\def \FF {\mathbb{F}}

\def \NN {\mathbb{N}}

\def \QQ {\mathbb{Q}}

\def \XX {\mathbb{X}}

\def \ZZ {\mathbb{Z}}

\def \Acal {\mathcal{A}}

\def \Gcal {\mathcal{G}}

\def \Ocal {\mathcal{O}}

\def \Xcal {\mathcal{X}}

\def \Cfr {\mathfrak{C}}

\def \Ifr {\mathfrak{I}}

\def \Mfr {\mathfrak{M}}

\def \Pfr {\mathfrak{P}}

\def \Tfr {\mathfrak{T}}

\def \Xfr {\mathfrak{X}}

\def \Zfr {\mathfrak{Z}}

\def \mfr {\mathfrak{m}}
\def \nfr {\mathfrak{n}}

\def \pfr {\mathfrak{p}}

\def \Ascr {\mathscr{A}}

\def \Jscr {\mathscr{J}}

\def \Lscr {\mathscr{L}}
\def \Mscr {\mathscr{M}}

\def \Oscr {\mathscr{O}}
\def \Pscr {\mathscr{P}}

\def \Sscr {\mathscr{S}}

\def \Ahat {\hat{A}}
\def \Bhat {\hat{B}}

\def \hbar {\bar{h}}

\def \Xtilde {\tilde{X}}

\def \rtilde {\tilde{r}}

\def \xtilde {\tilde{x}}

\def \Gsf {\mathbf{G}}

\def \bbf {\mathbf{b}}

\def \Esf  {\mathsf{E}}

\def \Gsf  {\mathsf{G}}

\def \Ksf  {\mathsf{K}}

\def \Ssf  {\mathsf{S}}

\def \Xsf  {\mathsf{X}}

\def \FFbar {\overline{\mathbb{F}}}

\def \QQbar {\overline{\mathbb{Q}}}

\def \BT    {{Barsotti-Tate group}}

\def \mono  {\hookrightarrow}

\def \epi   {\twoheadrightarrow}

\def \isom  {\stackrel{\sim}{\rightarrow}}

\newcommand{\pot}[1]{ [\hspace{-0,17em}[ {#1} ]\hspace{-0,17em}] }

\newcommand{\bigslant}[2]{{\raisebox{.2em}{$#1$}\hspace{-.3em}\left/ \hspace{-.2em}\raisebox{-.2em}{$#2$}\right.}}

\newcounter{subenvcounter}
\newenvironment{subenv}{%
 \begin{list}
  {\em (\arabic{subenvcounter})}
  {\setlength{\leftmargin}{20pt}
   \setlength{\rightmargin}{0pt}
   \setlength{\itemindent}{0pt}
   \setlength{\labelsep}{5pt}
   \setlength{\labelwidth}{13pt}
   \setlength{\listparindent}{\parindent}
   \setlength{\parsep}{0pt}
   \setlength{\itemsep}{0pt}
   \setlength{\topsep}{-\parskip}
   \usecounter{subenvcounter}}}
  {\end{list}}

\newcounter{asslistcounter}
\newenvironment{assertionlist}{
 \begin{list}
  {\upshape (\alph{asslistcounter})}
  {\setlength{\leftmargin}{18pt}
   \setlength{\rightmargin}{0pt}
   \setlength{\itemindent}{0pt}
   \setlength{\labelsep}{5pt}
   \setlength{\labelwidth}{13pt}
   \setlength{\listparindent}{\parindent}
   \setlength{\parsep}{0pt}
   \setlength{\itemsep}{0pt}
   \setlength{\topsep}{-.5\parskip}
   \usecounter{asslistcounter}}}
  {\end{list}}

\newenvironment{bulletlist}{
 \begin{list}
  {$\bullet$}
  {\setlength{\leftmargin}{18pt}
   \setlength{\rightmargin}{0pt}
   \setlength{\itemindent}{0pt}
   \setlength{\labelsep}{5pt}
   \setlength{\labelwidth}{13pt}
   \setlength{\listparindent}{\parindent}
   \setlength{\parsep}{0pt}
   \setlength{\itemsep}{0pt}
   \setlength{\topsep}{-.5\parskip}
   \usecounter{bulllistcounter}}}
 {\end{list}}

\def \Igfr {{\mathfrak{Ig}}}

\def \perf {{(p^{-\infty})}}
\def \univ {{\rm univ}}
\def \dom {{\rm dom}}

\def \Def {{\mathfrak{Def}}}
\def \loc {{\rm loc}}

\def \BTT {{Barsotti-Tate group with crystalline Tate tensors}}
\def \BTTs {{Barsotti-Tate groups with crystalline Tate tensors}}

\author[P.~Hamacher]{Paul~Hamacher}
 
\title[The product structure of Newton strata]{The product structure of Newton strata in the good reduction of Shimura varieties of Hodge type}

\address{Technische Universi\"at M\"unchen \\ Zentrum Mathematik - M11 \\ Boltzmannstra{\ss}e 3 \\ 85748 Garching bei M\"unchen \\ Deutschland}

\email{hamacher@ma.tum.de}

\begin{document}

 \begin{abstract}
   We construct a generalisation of Mantovan's almost product structure to Shimura varieties of Hodge type with hyperspecial level structure at $p$ and deduce that the perfection of the Newton strata are proétale locally isomorphic to the perfection of the product of a central leaf and a Rapoport-Zink space. The almost product formula can be extended to obtain an analogue of Caraiani's and Scholze's generalisation of the almost product structure for Shimura varieties of Hodge type. 
 \end{abstract}

 \maketitle

\section{Introduction}

 In the past years, Shimura varieties have been extensively studied in connection with the Langlands program. Conjecturally, the cohomology of Shimura varieties should decompose in terms of certain automorphic representations and their Langlands parameters. 
  A main tool to study the arithmetic of Shimura varieties is the Newton stratification, which is defined below. The almost product structure of Newton strata allows us to study their geometry and cohomology in terms of two simpler objects: Igusa varieties and Rapoport-Zink spaces. This is well understood in the special case of Shimura varieties of PEL-type. Building upon the work of Harris and Taylor \cite{HarrisTaylor:TheBook}, Mantovan gave the construction of the almost product structure of Newton strata for Shimura variety of PEL type with hyperspecial level structure at $p$ in \cite{Mantovan:Foliation} and derived a formula for the cohomology of a local system over Shimura varieties and Rapoport-Zink spaces associated to a representation of the linear algebraic group attached to the Shimura variety (\cite[Thm.~3.1]{Mantovan:PELnon-trivCoeff}). Recently, Caraiani and Scholze gave an infinite version of Mantovan's almost product structure in their paper \cite{CaraianiScholze:ShVar}. It gives a more elegant description at the cost that the objects considered are no longer of finite type. 

The aim if this paper is to extend this structure to the larger class of Shimura varieties of Hodge type.  Its applications to the cohomology of Shimura varieties of Hodge type will be discussed in a subsequent joint work with Wansu Kim. 
 To make our statements precise, we fix the notation as follows. Let $(\Gsf,\Xsf)$ be a Shimura datum of Hodge type, $\Ksf \subset \Gsf(\AA_f)$ a small enough compact open subgroup which is hyperspecial at $p$. Denote by $G$ the reductive model of $\Gsf$ over $\ZZ_p$ corresponding to $\Ksf_p$. The existence of a canonical integral model $\Sscr_G$ of the Shimura variety $\Sh_\Ksf(\Gsf,\Xsf)$ was shown by Kisin \cite{Kisin:IntModelAbType}, with some exceptions in the case $p=2$. His construction of $\Sscr_G$ also equips it with a principally polarised Abelian scheme $(\Acal_G,\lambda_G) \to \Sscr_G$ and crystalline Tate tensors $(t_{G,\alpha,x})$ on the (contravariant) Dieudonn\'e module $\DD(\Acal_{G,x}[p^\infty])$ for every point $x \in \Sscr_G(\FFbar_p)$. While the {\BTT} $(\Acal_{G,x}[p^\infty],\lambda_{G,x} (t_{G,\alpha,x}))$ depends on some choices made during the construction of $\Sscr_G$, it induces an isocrystal with $G$-structure $\Gcal_x$ over $\FFbar_p$ which is independent of them. Lovering constructed an isocrystal with $G$-structure over the special fibre $\Sscr_{G,0}$ of $\Sscr_G$ in \cite{lovering:GCrys}, which specialises to $\Gcal_x$ for every $x \in \Sscr_G(\FFbar_p)$. Thus we have a well-defined Newton stratification on $\Sscr_{G,0}$; we denote the Newton stratum associated to a $\sigma$-conjugacy class $\bbf$ in $G(L)$ by $\Sscr_{G,0}^\bbf$ (see section~\ref{ss Newton strat} for details). By the same argument the central leaves in $\Sscr$, i.e.\ the locus where the isomorphism class of $(\Acal_{G,x}[p^\infty],\lambda_{G,x}, (t_{G,\alpha,x}))$ is constant is well-defined.
 
 \subsection{The almost product structure in the special fibre}

 Using the almost product structure in deformation spaces constructed in \cite{Hamacher:DeforSpProdStr} as a starting point, we construct for a suitable central leaf $C \subset \Sscr_G$ a family of compatible surjective finite-to-finite correspondences
 \begin{center}
  \begin{tikzcd}
   & \Jscr_m^{(p^{-N})} \times \Mscr_{G,\mu}(\bbf)^{m_1,m_2} \arrow{dl} \arrow{dr}{\pi_N} & \\
   C^{(p^{-N})} \times \Mscr_{G,\mu}(\bbf)^{m_1,m_2} & & \Sscr_{G,0}^\bbf
  \end{tikzcd}
 \end{center}
 generalising Mantovan's construction in \cite{Mantovan:Foliation}. Here $(\Jscr_m)_{m\in\NN}$ denotes the tower of Igusa varieties over $C$ and $(\Mscr_{G,\mu}(\bbf)^{m_1,m_2})_{m_1 < m_2 \in \ZZ^2}$ a certain exhausting family of closed subschemes of the special fibre of the corresponding Rapoport-Zink space. We denote by $F_{p^N}\colon  \Jscr_m^{(p^{-N})} \to \Jscr_m$ the Frobenius and by $F_{p^\infty}\colon \Jscr_m^{\perf} \to \Jscr_m$ their limit for $N \to \infty$, i.e.\ the perfection. Moreover, we show that when taking the limit $N,m,m_1,m_2 \to \infty$ we obtain a correspondence of perfect schemes
 \begin{center}
  \begin{tikzcd}
   & \Jscr_\infty^{\perf} \times \Mscr_{G,\mu}(\bbf)^{\perf} \arrow{dl}[swap]{r_\infty} \arrow{dr}{\pi_\infty^\perf} & \\
   C^{\perf}\times \Mscr_{G,\mu}(\bbf)^{\perf} & & \Sscr_{G,0}^{\bbf,\perf}
  \end{tikzcd}
 \end{center}
 where $r_\infty^\perf$ and $\pi_\infty^\perf$ are weakly \'etale morphisms. More precisely, if $(X_0,\lambda_0,t_{0,\alpha})$ denotes a {\BTT} such that the central leaf $C$ is defined by $(\Acal_{G,x}[p^\infty],\lambda_G (t_{G,\alpha,x})) \cong (X_0,\lambda_0,t_{0,\alpha})$ for any $x \in C$, then $r_\infty^\perf$ is a torsor under the profinite group representing the automorphism group of $(X_0,\lambda_0,t_{0,\alpha})$ and $\pi_\infty^\perf$ is a torsor under the locally pro\'etale group representing its self-quasi-isogenies.
 
 \subsection{Caraiani-Scholze type product structure for Shimura varieties of Hodge type} 
 The infinite almost product structure can be extended to a generalisation of the product structure of Caraiani and Scholze constructed in \cite[\S~4]{CaraianiScholze:ShVar}. More explicitly, we construct an isomorphism \[\Igfr^\bbf \times \Mfr_{G,\mu}(\bbf) \isom \Xfr^\bbf \] of formal schemes over $\Spf \breve\ZZ_p$, where $\Igfr^b$ denotes the flat lift of $\Jscr_\infty^\perf$ to $\breve\ZZ_p$, $\Mfr_{G,\mu}(\bbf)$ is the corresponding Rapoport-Zink space and $\Xfr^\bbf$ parametrises trivialisations of $(\Acal_G,\lambda_G,(t_{G,\alpha}))$ up to isogeny.
 
 
 The structure of this article is as follows. After giving the necessary background in section~\ref{sect background}, we lay the foundations of our construction by recalling Mantovan's almost product structure in the special case of the Siegel variety in section~\ref{sect siegel}. The main construction is given in section~\ref{sect mantovan}, where we construct the almost product structure of Newton strata in Shimura varieties of Hodge type with hyperspecial level at $p$. Building upon this construction, we also generalise the Caraiani-Scholze product structure in section~\ref{sect caraiani scholze}.

 This article was written independently from the article \cite{zhang} of Zhang, which was put on ArXiv approximately two weeks before the first version of this paper. In his work, he  gives a partial construction of the morphisms $\pi_N\colon \Jscr_m^{(p^{-N})} \times \Mscr_{G,\mu}(\bbf)^{m_1,m_2} \to \Sscr_{G,0}$. He constructs the Igusa tower and the restriction of $\pi_N$ to the smooth locus.
 
 \emph{Acknowledgements:} I am grateful to Mark Kisin for many helpful discussions and his advice. I warmly thank Thomas Lovering for giving me a preliminary version of his thesis. I thank Stephan Neupert and Eva Viehmann for pointing out some mistakes in the preliminary version of this article. Most of this work was written during a stay at the Harvard University which was supported by a fellowship within the Postdoc program of the German Academic Exchange Service (DAAD). I want to thank the Harvard University for its hospitality. Moreover, the author was partially supported by the ERC starting grant 277889 ``Moduli spaces of local $G$-shtukas''.

\section{Shimura varieties of Hodge type} \label{sect background}

 Let $p > 2$ be a prime number. We will focus our study on the class of Shimura varieties of Hodge type with hyperspecial level structure at $p$.  
  
 \subsection{Crystalline Tate tensors} 

 While the reduction modulo $p$ of a Shimura variety of Hodge type will not be given by a moduli description, we will have an Abelian variety with additional structure by crystalline Tate tensors which can be regarded as a replacement for a universal object. We briefly recall the notion of crystalline Tate tensors as presented in \cite{Kim:RZ}. 

 \begin{notation}
  Assume that $M$ is either
  \begin{bulletlist}
   \item a locally free module over a ring $R$,   
   \item a locally free $\Ocal_{S/\ZZ_p}$-crystal for an $\FF_p$-scheme $\ZZ_p$ or
   \item an isocrystal associated to a locally free $\Ocal_{S/\ZZ_p}$-crystal.
  \end{bulletlist}
  Then $M^\otimes$ is defined as the direct sum of any finite combination of tensor products, symmetric products, alternating products and duals of $M$.
  We call tensor of $M$ is a morphism $s\colon \mathbf{1} \to M^\otimes$, where $\mathbf{1}$ denotes $R$, $\Ocal_{S/\ZZ_p}$ or $\Ocal_{S/\ZZ_p}[\frac{1}{p}]$.
  
 \end{notation}

 We fix an $\FF_p$-scheme $S$ as base scheme and consider the category of quasi-coherent $\Oscr_{S/\ZZ_p}$-crystals. Then $\mathbf{1} = \Oscr_{S/\ZZ_p}$ comes with a canonical Frobenius action $F\colon  \Oscr_{S/\ZZ_p}^{(p)} \to \Oscr_{S/\ZZ_p}$.
  For any Abelian scheme or {\BT} $X$ over $S$ we denote the contravariant Dieudonn\'e crystal as in \cite[\S~3]{BerthelotBreenMessing:Book} by $\DD(X)$. The pull-back $\DD(X)_S$ to the Zariski site of $S$ is naturally equipped with the Hodge filtration $\Fil_{\DD(X)}^1 := \omega_X \subset \DD(X)_S$, which is Zariski-locally a direct summand of rank $\dim X$. Moreover the relative Frobenius of $X$ over $S$ induces a map $F\colon \DD(X)^{(p)} \to \DD(X)$, which is also called the Frobenius. As the relative Frobenius is an isogeny, $F$ induces an isomorphism of isocrystals $\DD(X)[\frac{1}{p}]^{(p)} \isom \DD(X)[\frac{1}{p}]$ and $\DD(X)[\frac{1}{p}]^{\otimes\, (p)} \isom \DD(X)[\frac{1}{p}]^\otimes$.
 \begin{definition}[{\cite[Def.~2.3.4]{Kim:RZ}}]
  A tensor $t$ of $\DD(X)$ is called a crystalline Tate tensor if it induces a morphism of $F$-isocrystals $\mathbf{1} \to \DD(X)[\frac{1}{p}]^\otimes$, i.e. it is Frobenius equivariant.
 \end{definition}

 \subsection{Construction of integral models} \label{sect integral model}
  
 Let $(\Gsf,\Xsf)$ be Shimura datum of Hodge type, i.e.\ there exists an embedding of Shimura data $(\Gsf,\Xsf) \mono (\mathsf{GSp}_{2g},\Ssf^\pm)$ or some integer $g$ where $\Ssf^\pm$ denotes the Siegel double space. Now $\Xsf$ defines a conjugacy class $\{\mu\}$ of cocharacters of $\Gsf$. Its field of definition $\Esf$ is called the Shimura field of $(\Gsf,\Xsf)$. We denote by $E$ its $p$-adic completion for some fixed embedding $\overline\QQ \mono \overline\QQ_p$ and $O_E$ the corresponding ring of integers. We fix a small enough compact open subgroup $\Ksf = \Ksf_p\Ksf^p \subset \Gsf(\AA_f)$ which is hyperspecial at $p$ (for the precise meaning of ``small enough'' see for example \cite[\S~2]{Milne:CanModelAnnArbor}) and denote by $\Sh_\Ksf(\Gsf,\Xsf)$ the canoncial model of the associated Shimura variety.  The existence of a canonical integral model over $O_E$ for $\Sh_\Ksf(\Gsf,\Xsf)$ was shown by Mark Kisin (\cite{Kisin:IntModelAbType}). We briefly recall his construction and introduce the notions which will be fundamental for the study of the special fibre which follows.
    
 Let $G$ be the Bruhat-Tits-group scheme associated to $\Ksf_p$. As $\Ksf_p$ is hyperspecial, $G$ is a reductive group scheme over $\ZZ_p$. We fix a (unique up to isomorphism) reductive group scheme $G_{\ZZ_{(p)}}$ over $\ZZ_{(p)}$ such that $G = G_{\ZZ_{(p)}} \otimes \ZZ_p$. By \cite[\S~1.3.3]{Kisin:LanglandsRapoport} there exists an embedding of Shimura data $\iota\colon (\Gsf,X) \mono (\mathsf{GSp}_{2g},S^\pm)$ which is induced by an embedding $G_{\ZZ_{(p)}} \hookrightarrow \GSp_{\ZZ_{(p)}}$. Moreover, by \cite[Lemma~2.1.2]{Kisin:IntModelAbType} there exists an open compact subgroup $\Ksf'^p\subset \mathsf{GSp}_{2g}(\AA_f^p)$ such that $\iota$ induces an embedding  $\Sh_\Ksf(\Gsf,X) \mono \Sh_{\Ksf'}(\mathsf{GSp}_{2g},S^\pm)$ with $\Ksf' = \Ksf_p '\Ksf'^p$ and $\Ksf'_p = \GSp_{2g}(\ZZ_p)$.
 
 Denote by $\Ascr_g$ the moduli space of principally polarised Abelian varieties with $\Ksf'^p$-level structure over $\ZZ_{(p)}$ and let $(\Acal^\univ,\lambda^\univ,\eta^\univ)$ be its universal object. Then $\Ascr_g$ is a smooth scheme whose generic fibre is canonically isomorphic to $\Sh_{\Ksf'}(\mathsf{GSp}_{2g},S^\pm)$ (see e.g.\ \cite[\S~5]{Kottwitz:PtShimuraVarFinFields}).  Let $\Sscr_G^-$ be the closure of $\Sh_{\Ksf}(\Gsf,\Xsf)$ in $\Ascr_g \otimes O_E$, denote by $\Sscr_G$ its normalisation. 
 
 \begin{theorem}[{\cite[Thm.~2.3.8]{Kisin:IntModelAbType}}]
  $\Sscr_G$ is the canonical integral model of $\Sh_{\Ksf}(\Gsf,\Xsf)$, in particular it is smooth and independent of the choice of the embedding $\Sh_\Ksf(\Gsf,\Xsf) \mono \Sh_{\Ksf'}(\mathsf{GSp}_{2g},\Ssf^\pm)$.
 \end{theorem}
 
 A central point of his proof is the following construction. Denote by $h:(\Acal_G,\lambda_G) \to \Sscr_G$ the pullback of $(\Acal^\univ,\lambda^\univ)$. Let $(M,\psi)$ be the symplectic $\ZZ_p$ module of rank $2g$ and fix  a family of tensors $(s_\alpha) \subset M^\otimes$ such that $G \subset \GSp_{2g}$ is their stabiliser (cf. \cite[Prop.~1.3.2]{Kisin:IntModelAbType}). For every closed geometric point in the special fibre $x \in \Sscr_G(\FFbar_p)$, Kisin associates a family of crystalline Tate tensors $(t_{G,\alpha,x}) \subset \DD(\Acal_{G,x})$ such that $(\DD(\Acal_{G,x}),\lambda_{G,x}, t_{G,\alpha,x})$ is isomorphic to $(M_{\breve\ZZ_p},\psi \otimes 1, s_\alpha \otimes 1)$ (cf.\ \cite[Prop.~1.3.7]{Kisin:LanglandsRapoport}). 
 
 One can interpolate these pointwise tensors with global tensors as follows. Recall that the choice of $(s_\alpha)$ induced horizontal tensors $(t_{G,\alpha,{\rm dR}})$ on the first de Rham - cohomology $R^1h_\ast\Omega^\bullet$, as well as tensors $(t_{G,\alpha,p})$ on $R^1h_{\eta\, et \ast} \ZZ_p$ and $(t_{G,\alpha,l})$ on $R^1h_{et \ast} \QQ_l$ for $l \not= p$ (cf.\ \cite[\S~1.3.6]{Kisin:LanglandsRapoport}). Since $\Sscr_G$ is smooth, the tensors $t_{G,\alpha,{\rm dR}}$ induce crystalline Tate tensors $t_{G,\alpha}$ on $\DD(\Acal_G/\Sscr_{G,0})$ for every $\alpha$ (see \cite[\S~3.1.6]{HowardPappas:GSpin}). By construction the family $(t_{G,\alpha})$ specialises to $(t_{G,\alpha,x})$ for every $x \in \Sscr_G(k)$.
 
 Using the crystalline Tate tensors introduce above, we can describe the local geometry of the normalisation morphism $\Sscr_G \to \Sscr_G^-$ as follows.

 \begin{proposition}[{\cite[Prop.~1.3.9, Cor.~1.3.11]{Kisin:LanglandsRapoport}}] \label{prop local-global}
  Let $x_0 \in \Sscr_G^-(\FFbar_p)$.
  \begin{subenv}
   \item Two points $x,x' \in \Sscr_G(\FFbar_p)$ over $x_0$ are identical if and only if $t_{G,\alpha,x} = t_{G,\alpha,x'}$ for every $\alpha$.
   \item The morphism mapping the formal neighbourhood $\Sscr_{G,x}^\wedge$ to the deformation space $\Def(\Acal_x^\univ[p^\infty],\lambda_x^\univ, t_{G,\alpha,x})$ induced by the deformation $(\Acal_G, \lambda_G, (t_{G,\alpha}))|_{\Sscr_{G,x}^\wedge}$ is an isomorphism. In particular, we have a commutative diagram
   \begin{center}
    \begin{tikzcd}[column sep = small]
     \Sscr_{G,x_0}^\wedge \arrow{d} \arrow{r}{\sim} & \coprod\limits_{x \mapsto x_0} \Def(\Acal_{x_0}^\univ[p^\infty],\lambda_{x_0}^\univ, t_{G,\alpha,x}) \arrow{d} & \\
     \Sscr_{G,x_0}^{- \wedge} \arrow{r}{\sim} & \bigcup\limits_{x \mapsto x_0} \Def(\Acal_{x_0}^\univ[p^\infty],\lambda_{x_0}^\univ, t_{G,\alpha,x})  \arrow[hook]{r} & \Def(\Acal_{x_0}^\univ[p^\infty],\lambda_{x_0}^\univ) = \Ascr_{g,x_0}^\wedge.
    \end{tikzcd}
   \end{center}
  \end{subenv}
 \end{proposition}

\subsection{The Newton stratification and Oort's foliation}
\label{ss Newton strat}
  In order to define the Newton stratification on the special fibre $\Sscr_{G,0}$ of $\Sscr_G$, we first recall some facts about $\sigma$-conjugacy classes. We fix a maximal torus and a Borel subgroup $T \subset B \subset G$. Let $k$ be an algebraically closed field of characteristic $p$, $W(k)$ its ring of Witt vectors, $L(k)$ be the fraction field of $W(k)$ and let $\sigma$ be the Frobenius automorphism of $L(k)$. We denote for any element $b \in G(L(k))$ by $[b]$ its $\sigma$-conjugacy class and by $B(G)$ the set of $\sigma$-conjugacy classes in $G(L(k))$. The set $B(G)$ is independent of the choice of $k$ by \cite[Lemma~1.3]{RapoportRichartz:Gisoc}. An element $\bbf \in B(G)$ is uniquely determined by two invariants (\cite[\S~6]{Kottwitz:Gisoc1}), its Newton point $\nu(\bbf) \in X_\ast(T)^{\Gal(\QQbar_p,\QQ_p)}_\dom$ and its Kottwitz point $\kappa(\bbf) \in \pi_1(G)_{\Gal(\QQbar_p,\QQ_p)}$.

  For $x \in \Sscr_G(\FFbar_p)$ choose $b_x$ such that the Frobenius morphism on $\DD(\Acal_x^\univ)$ is mapped to $b_x \sigma$ under an isomorphism $(\DD(\Acal_x^\univ),\lambda_{G,x},t_{\alpha,x}) \cong (M_{\breve\ZZ_p},\psi \otimes 1, s_\alpha \otimes 1)$. Then $b_x$ is uniquely defined up to $G(\breve\ZZ_p)$-$\sigma$-conjugacy, in particular $[b_x]$ is independent of the choice of an isomorphism. We define for any $\bbf \in B(G)$ 
  \[
   \Sscr_G^\bbf(\FFbar_p) = \{ x\in \Sscr_{G,0}(\FFbar_p) \mid b_x \in \bbf \}.
  \]
  Lovering showed in \cite{lovering:GCrys} that since $(\DD(\Acal_G),\lambda_G,t_{G,\alpha})$ is locally isomorphic to $(M,\psi,s_\alpha)$, we obtain an isocrystal with $G$-structure over $\Sscr_{G,0}$ in the sense of Rapoport and Richartz after inverting $p$. Thus, the $\Sscr_G^\bbf(\FFbar_p)$ are locally closed subvarieties of $\Sscr_{G,0} (\FFbar_p) $ by \cite[Thm.~3.6]{RapoportRichartz:Gisoc}; they are called the Newton strata of $\Sscr_{G,0}$.

\begin{remark*}
 The above description of $\Sscr_G^\bbf$ is valid for any algebraically closed field of characteristic $p$. More precisely, for any geometric point $x \in \Sscr_{G,0}(k)$, one can define $b_x \in G(L(k))$ (unique up to $G(W(k)$-conjugacy) analogously to the case $k= \FFbar_p$ above. By definition of the Newton stratification, the geometric points of a stratum $\Sscr^\bbf$ are given by
 \[
   \Sscr_G^\bbf(k) = \{ x\in \Sscr_{G,0}(k) \mid b_x \in \bbf \}.
  \]
 However, in the rest of this article we work almost exclusively with \emph{closed} geometric points.
\end{remark*}

  One can describe the elements $b_x$ obtained by above construction as follows. By \cite[\S~1.4.1]{Kisin:LanglandsRapoport}, we have $b_x \in G(\breve\ZZ_p)\mu(p)G(\breve\ZZ_p)$, where $\mu \in X_\ast(T)$ is the unique dominant cocharacter in the conjugacy class $\{\mu\}$ defined at the beginning of section~\ref{sect integral model}. On the other hand, by \cite[Prop.~1.4.4]{Kisin:LanglandsRapoport} there exists for any $b \in G(\breve\ZZ_p)\mu(p)G(\breve\ZZ_p) \cap \bbf$ with $\Sscr_G^\bbf \not= \emptyset$ an $x \in \Sscr_{G}^\bbf(\FFbar_p)$ such that $b_x = b$ (up to $G(\breve\ZZ_p)$-$\sigma$-conjugacy).
  
  Let $\bbf \in B(G)$ such that $\Sscr_G^\bbf$ is nonempty and fix a polarised {\BTT} $(X_0,\lambda_0,t_{0,\alpha})$ over $k$ such that its Dieudonn\'e crystal is isomorphic to $(M_{\breve\ZZ_p},\psi\otimes 1, s_\alpha \otimes 1)$ and its Frobenius corresponds to an element of $\bbf \cap G(\breve\ZZ_p)\mu(p)G(\breve\ZZ_p)$. The corresponding central leaf $C$ in $\Sscr_G^\bbf$ is the set of geometric points $x\colon \Spec k \to \Sscr_G^\bbf$ such that there exists an isomorphism
  \[
   (\Acal_x^\univ[p^\infty],\lambda_x^\univ, t_{\alpha,x}) \cong (X_0, \lambda_0,t_{0,\alpha}) \otimes k.
  \]
  Then $C$ defines a closed subset of $\Sscr_G^\bbf$ by \cite[Prop.~2.14]{Hamacher:DeforSpProdStr}.

  \begin{proposition}
   $C$ is smooth of dimension $2 \langle \rho, \nu(\bbf) \rangle$, where $\rho$ denotes the half-sum of positive roots of $G$.
  \end{proposition}
  \begin{proof}
    We fix $x \in C(\FFbar_p)$ and consider the isomorphism $\Sscr_{G,0,x}^\wedge \cong \Def(X_0,\lambda_0,t_{0,\alpha})$. As this identifies the canonical {\BTTs} lying over these schemes, it identifies $C_x^\wedge$ with the central leaf in the deformation space. Thus all formal neighbourhoods of closed points of $C$ are isomorphic, which shows that it is smooth. Also, $\dim C$ equals the dimension of the central leaf in $\Def(X_0,\lambda_0,t_{0,\alpha})$, thus the second assertion follows from \cite[Prop.~3.9]{Hamacher:DeforSpProdStr}.
  \end{proof}
    
  From now on, we fix $\bbf$ and assume that $X_0$ is completely slope divisible. We refer the reader to \cite[\S~1]{OortZink:pdiv} for details on completely slope divisible \BT s. Let $\bbf'$ denote the image of $\bbf$ in $B(\GSp_{2g})$, $\mu'$ denote the image of $\mu$ in $X_\ast(\GSp_{2g})$ and let $C' \subset \Ascr_g^{\bbf'}$ denote the central leaf associated to $(X_0,\lambda_0)$. Since $C'$ is smooth the {\BT} $\Xcal_0 := \Acal^\univ[p^\infty]_{|C'}$ is completely slope divisible (\cite{Mantovan:Foliation}~\S~3). We denote by $X_0^i$ and $\Xcal_0^i$ the respective isoclinic subquotients of $X_0$ and $\Xcal_0$.
 
\section{Almost product structure in the Siegel moduli space}
\label{sect siegel}
 As our construction of the almost product structure for Shimura varieties of Hodge type will heavily rely on the almost product structure for Siegel moduli spaces, we recall its construction.

 \subsection{Construction of the almost product structure}
 The almost product structure of Newton strata in the Siegel moduli space is originally due to Oort \cite{Oort:Foliations}. Mantovan generalized this result to Shimura varieties of PEL-type and gave a more rigorous description in \cite{Mantovan:Foliation}. We briefly sketch Mantovan's construction in the case of the Siegel moduli space. In particular, no claim of originality is made for this subsection.
 
 Denote by $\mathscr{J}'_m \to C'$ the Igusa variety of level $m$, that is the scheme parametrising families of isomorphisms $j_{m}^i\colon X_0^i[p^m] \isom \Xcal_0^i[p^m]$ which commute with the polarisations and for any $m' > m$ can be lifted \'etale locally to an isomorphism of the schemes of $p^{m'}$-torsion points. 
 
 \begin{proposition}[\cite{Mantovan:Foliation}~Prop.~4] \label{prop igusa}
  The morphism $\Jscr'_m \to C'$ is finite \'etale and Galois with Galois group $\Gamma'_{b,m} := \image(\Aut(X_0) \to \Aut(X_0[p^m]))$.
 \end{proposition}
 
 Denote by $\Mfr_{\GSp_{2g},\mu'}(\bbf')$ the Rapoport-Zink space parametrising quasi-isogenies with source $(X_0,\lambda_0)$ and by $\Mscr_{\GSp_{2g},\mu'}(\bbf')$ its underlying reduced subscheme. For any pair of integers $m_1,m_2$, let $ \Mscr_{\GSp_{2g},\mu'}(\bbf')^{m_1,m_2} \subset \Mscr_{\GSp_{2g},\mu'}(\bbf')$ denote the (closed) subvariety parametrising quasi-isogenies $\rho$ such that $p^{m_1}\rho$ and $p^{m_2}\rho^{-1}$ are isogenies.
  
 Now assume $m_1 + m_2 \leq m$ and choose $N$ such that the filtration on $F_{p^N/k}^\ast \Xcal_0[p^m]$ induced by the slope filtration splits canonically as in \cite{Mantovan:Foliation}~Lemma~8. Thus the family of isomorphisms $(j_m^i)$ induced by $\Jscr_m'^{(p^{-N})} \to \Jscr_m'$ yields an isomorphism $j_m\colon X_0[p^m] \times \Jscr_m'^{(p^{-N})} \isom \Xcal_{0}[p^m]_{\Jscr_m'^{(p^{-N})}}$.
 
 Now $\pi_N'\colon \Jscr_m'^{(p^{-N})} \times \Mscr_{\GSp_{2g},\mu'}(\bbf')^{m_1,m_2} \to \Ascr_g^{b'}$ is constructed as follows. Let $(\Acal,\lambda,\eta)$ be the pullback of $(\Acal^\univ,\lambda^\univ,\eta^\univ)$ to $\Jscr_m'^{(p^{-N})} \times \Mscr_{\GSp_{2g},\mu'}(\bbf')^{m_1,m_2}$. Denote by $\rho$ the universal quasi-isogeny over $\Mscr_{\GSp_{2g},\mu'}(b')^{m_1,m_2}$. We define
 \[
  \Acal' := \bigslant{\Acal}{j_m(\ker (p^{m_1}\rho))}.
 \] 
 and endow it with the polarisation $\lambda'$ and a $K'^p$-level structure $\eta'^p$ induced by $\lambda$ and $\eta$, respectively. Then $(\Acal',\lambda',\eta')$ defines a morphism
 \[
  \pi_{N,m,m_1,m_2}'\colon \Jscr_m'^{(p^{-N})} \times \Mscr_{\GSp_{2g},\mu'}(\bbf')^{m_1,m_2} \to \Ascr_g^{\bbf'}.
 \]
 The $\pi_{N,m,m_1,m_2}$ satisfy the obvious commutativity relations when varying $N,m,m_1,m_2$.
  
 \begin{proposition}[\cite{Mantovan:Foliation}~Prop.~9] \label{prop compatibility properties of mantovans morphism}
  The family of morphisms $(\pi_{N,m,m_1,m_2}')$ satisfies the obvious commutativity relations. More precisely, the following identities hold.
  \begin{subenv}
   \item $\pi_{N,m,m_1,m_2}' \circ (F_{p/k} \times \id) = \pi_{N+1,m,m_1,m_2}'$.
   \item Let $r_{m+1,m}'\colon \Jscr_{m+1}' \epi \Jscr_m'$ be the canonical projection. Then $\pi_{N,m-1,m_1,m_2}' \circ (r_{m+1,m}' \times \id) = \pi_{N,m,m_1,m_2}'$, where it is assumed that $m_1+m_2 \leq m-1$.
   \item Let $\iota'\colon \Mscr_{\GSp_{2g},\mu'}(\bbf)^{m_1',m_2'} \mono \Mscr_{\GSp_{2g},\mu'}(\bbf)^{m_1,m_2}$ the obvious embedding for $m_1' \leq m_1, m_2' \leq m_2$. Then $\pi_{N,m,m_1,m_2}' \circ \iota = \pi_{N,m',m_1',m_2'}'$.
  \end{subenv}
 \end{proposition}

 \begin{notation}
  For the rest of this article, we will abbreviate $\pi_{N,m,m_1,m_2}'$ by $\pi_{N,m}'$ and assume that one has made a fixed choice of $m_1,m_2$ depending on $m$ such that for varying $m$ the sequences  $(m_1)$ and $(m_2)$ are monotonously increasing and unbounded.
 \end{notation}
 
 \begin{proposition}[{\cite[Prop.~10,~11]{Mantovan:Foliation}}]
 \label{prop geometric properties of mantovans morphism}
  The morphism $\pi_{N,m}'$ is finite, and surjective if $m$ is big enough.
 \end{proposition}
 
 \subsection{The almost product structure of infinite level}
 Proposition~\ref{prop compatibility properties of mantovans morphism} tells us that we can take the limit for $N,m \to \infty$ to obtain a morphism
 \[
  \pi_\infty'\colon \underleftarrow{\lim} \Jscr_m^{\perf} \times \Mscr_{\GSp,\mu'}(\bbf') \to \Ascr_g^{\bbf'}.
 \]
 A generalisation of this infinite almost product structure was studied by Caraiani and Scholze in \cite[\S~4]{CaraianiScholze:ShVar}. The results in this section are a formal consequence of their work, we will give the proofs where they simplify due to our restriction to the special fibre or where we will need the technical details to generalise to Shimura varieties of Hodge type.

We can define the morphism $\pi_\infty'$ directly rather than taking a limit. We define the Igusa variety of infinite level as $\Jscr_\infty' := \underleftarrow{\lim} \Jscr_m'$ and by $r_\infty' \coloneqq \varprojlim r_m'\colon \Lscr_\infty' \to C'$ the canonical projection. As a consequence of the moduli description of the finite level Igusa varieties, we obtain that $\Jscr_\infty'$ parametrises families of isomorphisms $X_0^i \isom \Xcal_0^i$. We denote by $j_\infty^{i, \univ}$ the universal isomorphisms over $\Jscr_\infty'$. As the slope filtration splits canonically over perfect schemes, the $j_\infty^{i, \univ}$s induce a universal isomorphism
 \[
  j = \oplus j_\infty^{i, \univ}\colon X_0 \times \Jscr_\infty'^{\perf} \isom \Xcal_{0, \Jscr_\infty'^{\perf}}
 \]
 
\begin{lemma}[{\cite[Prop.~4.3.8]{CaraianiScholze:ShVar}}] \label{lemma Igusa moduli}
 $\Jscr_\infty'^\perf$ is the moduli space of isomorphisms $(X_0,\lambda) \isom (\Acal^\univ[p^\infty],\lambda^\univ)$ and $j$ is the universal object.
\end{lemma} 
  
 Thus we can repeat the above construction. Let $\rho^\univ$ be the universal quasi-isogeny over $\Mscr_{\GSp_{2g},\mu'}(\bbf')$ and $(\Acal,\lambda,\eta)$ the pullback of $(\Acal^\univ,\lambda^\univ,\eta^\univ)$ to $\Jscr_\infty'^\perf \times \Mscr_{\GSp_{2g},\mu'}(\bbf')$. Now, Zariski-locally there exists an integer $m_1$ such that $p^{m_1}\rho^\univ$ is an isogeny. By glueing 
 $
  \Acal/j(\ker p^{m_1} \rho)
 $
 over a suitable Zariski covering, we obtain a polarised Abelian variety over $\Jscr_\infty'^{\perf} \times \Mscr_{\GSp_{2g},\mu'}(\bbf')$ with $\Ksf'^p$ level structure. The induced a morphism
 \[
  \Jscr_\infty'^{\perf} \times \Mscr_{\GSp_{2g},\mu'}(\bbf') \to {\Ascr_g^{\bbf'}}.
 \]
 is identical to $\pi_\infty'$ since their moduli descriptions coincide.
 
 In order to describe the geometric properties of the almost product structure, we need to derive some results about constant {\BT}s, i.e.\ {\BT}s which descend to the spectrum of an algebraically closed field. In the following let $k$ be an algebraically closed field of characteristic $p$.

 \begin{lemma}
  Let $X,Y$ be {\BT s} over $k$ and let $S$ be an integral $k$-scheme. Then the morphism $\Hom(X,Y) \to \Hom(X_S,Y_S)$ induced by pullback is an isomorphism.
 \end{lemma}
 \begin{proof}

  The injectivity of the morphism is a direct consequence of fpqc descent. To prove surjectivity, let $\varphi \in \Hom(X_S,Y_S)$. In the case where $S = \Spec k_0$ is the spectrum of a perfect field, it follows from Dieudonn\' e theory that $\varphi$ is defined over $k$ (see for example \cite[Lemma~3.9]{RapoportRichartz:Gisoc}). If $k_0$ is a non-perfect field, then there exists a $\varphi_0 \in \Hom(X,Y)$ such that after base change to the perfection $k_0^\perf$ we have $\varphi_{k_0^\perf} = \varphi_{0,k_0^\perf}$. It follows by fpqc-descent that $\varphi = \varphi_{0,k}$. To prove surjectivity in full generality, let $\eta$ be the generic point of $S$ and let $\varphi_0 \in \Hom(X,Y)$ such that $\varphi_\eta = \varphi_{0,\eta}$. Since the property of two homomorphisms being equal is closed, it follows that $\varphi = \varphi_{0,S}$.
 
 

 \end{proof}
 
 We can deduce the homomorphisms of constant {\BT}s over not necessarily integral schemes from the previous lemma, but still need that the base scheme is reduced.

 \begin{lemma}
  Let $X,Y$ be {\BT s} over $k$ and let $H := \Hom(X,Y)$ as topological group equipped with the $p$-adic topology. Then the functor on reduced $k$-schemes
  \begin{eqnarray*}
   ({\rm RedSch}_k)^{opp} &\to& ({\rm Sets}) \\
   S &\mapsto& \Hom(X_S,Y_S)
  \end{eqnarray*}
  is represented by the locally constant $k$-scheme induced by $H$.
 \end{lemma}
 \begin{proof}
  Let $S$ be a reduced $k$-scheme. Giving a morphism $\phi \in Hom(X_S,Y_S)$ is equivalent to giving an inductive system of $\phi_m \in \Hom(X[p^m],Y[p^m])$ for $m > 0$. By the previous lemma $\phi_m$ can only be lifted to a morphism of {\BT}s only if it is constant on every irreducible component and is hence given by a morphism
  \[
   g_m:S \to \left( \bigslant{H}{p^m} \right)_k \subset \underline\Hom(X[p^m],Y[p^m])
  \]
  or equivalently by a continuous map 
  \[
   f_m\colon \pi_0(S)  \to \bigslant{H}{p^m}.
  \]
  Altogether, we have a natural bijection
  \[
   \Hom(X_S,Y_S) \cong \underrightarrow{\lim} \Hom_{\rm cont}(\pi_0(S), H/p^m) = \Hom_{\rm cont}(\pi_0(S),H).
  \]
 \end{proof}

 \begin{remark*}
  There can be a lot more homomorphisms if $S$ is not reduced, see \cite[\S~4.1]{CaraianiScholze:ShVar}.
 \end{remark*}

 The following is immediate.

 \begin{corollary}
  Let $X,Y$ be {\BT s} over $k$ and let $J := \QIsog(X,Y)$ be equipped with the $p$-adic topology. Then the functor
  \begin{eqnarray*}
   ({\rm RedSch}_k)^{opp} &\to & ({\rm Sets}) \\
   S &\mapsto& \QIsog(X_S,Y_S)
  \end{eqnarray*}
  is represented by the locally constant $k$-scheme induced by $J$.
 \end{corollary}
 
 Denote by $\Gamma_{b'}$ the automorphism group of $(X_0,\lambda_0)$ and by $J_{b'}$ the group of self-quasi isogenies equipped with the $p$-adic topology. When regarded as locally profinite $\FFbar_p$-group schemes, they represent the sheaves $\underline{\Aut}(X_0,\lambda_0)$ and $\underline{\Aut}_{\QQ}(X_0,\lambda_0)$, respectively, in the category of reduced schemes by the previous proposition.

 The group $J_{b'}$ acts on $\Jscr_\infty'^{\perf}$ via
 \[
  g.(\Acal;j^1,\ldots,j^r) = ((j_\ast g^{-1})(\Acal,\lambda,\eta);g^\ast j).
 \] 
 More explicitly, $(j_\ast g)(\Acal,\lambda,\eta)$ is defined as $\bigslant{\Acal}{j(\ker p^{m_2}g^{-1})}$ with the induced additional structure where $m_2$ is big enough such that $p^{m_2}g^{-1}$ is an isogeny. The isomorphism $g^\ast j$ is defined as the (unique) morphism such that the following diagram is commutative.
 \begin{center}
  \begin{tikzcd}
   X_0 \arrow{r}{j} \arrow[two heads]{d}{p^{m_2} g^{-1}} & \Acal[p^\infty] \arrow{d} \\
   \bigslant{X_0}{j(\ker p^{m_2}g^{-1})} = X_0 \arrow[dashed]{r}{g^\ast j} & (\bigslant{\Acal}{j(\ker p^{m_2}g^{-1})})[p^\infty].
  \end{tikzcd}
 \end{center} 
 Because of our moduli interpretation of $J_{b'}$, this action is continuous.

 \begin{remark}
  \begin{subenv}  
   \item If $\rho \in \Gamma_{b'}$, then
   \[
    \rho.(\Acal;j) = (\Acal,j \circ g)
   \]
   \item The $J_{b'}$-action is induced by the $\Aut(\widetilde\XX_b)$-action given in \cite[Cor.~4.3.5]{CaraianiScholze:ShVar}. In particular, the $J_{b'}$-action extends the action of the submonoid $S_{b'}$ on the tower of (finite level) Igusa varieties as constructed in \cite{Mantovan:Foliation}~\S~4. Moreover Mantovan extends the action of $S_{b'}$ to $J_{b'}$ on the cohomology, where her and our $J_{b'}$-action coincide.
  \end{subenv}
 \end{remark}
 
 It is favourable for our purposes to consider the perfection $\pi_\infty'^{\perf}$ of $\pi_\infty'$; we prove in the proposition below that this makes the morphism weakly \'etale. More precisely, consider the diagram
 \begin{center}
  \begin{tikzcd}
    & \Jscr_\infty'^{\perf} \times \Mscr_{\GSp_{2g},\mu'}(\bbf')^{\perf} \arrow{dl}[swap]{r_\infty'^{\perf} \times \id} \arrow{dr}{\pi_\infty'^{\perf}} & \\
    C'^{\perf} \times \Mscr_{\GSp_{2g},\mu'}(\bbf')^{\perf} & & {\Ascr_g^{\bbf'}}^{\perf}
  \end{tikzcd}
 \end{center}
 where $(-)^\perf$ denotes the perfection of $(-)$ over $\FFbar_p$. Note that by construction $r_\infty'$ is $\Gamma_{b'}$-invariant and $\pi_\infty'$ is $J_{b'}$-invariant.  
 
  \begin{proposition}[cf.\ {\cite[Prop.~4.3.13]{CaraianiScholze:ShVar}}] \label{prop almost product siegel infinite level}
  We have the following results on the geometry of above correspondence.
  \begin{subenv}
   \item $r_\infty'^{\perf}$ is profinite pro\'etale with pro-Galois group $\Gamma_{b'}$.
   \item $\pi_\infty'^{\perf}$ is a $J_{b'}$-torsor for the pro\'etale topology.
  \end{subenv}
 \end{proposition}
 \begin{proof}
  The first assertion is a direct consequence of Proposition~\ref{prop igusa}. To prove the second assertion, we first prove that $\pi_\infty'$ is weakly \'etale. As this property is local on the source, it suffices to show that $\Jscr_\infty'^{\perf} \times U \to {\Ascr_g^{b'}}^{\perf}$ is weakly \'etale for any  quasi-compact open subset $U$. As $U$ is quasi-compact, there exist $m_1,m_2$ such that $U \subset \Mscr_{\GSp_{2g},\mu'}(b')^{m_1,m_2 \,\perf}$. Thus the above morphism factors as
  \[
   \Jscr_\infty'^{\perf} \times U \to \Jscr_{m_1+m_2}'^{\perf} \times U \to {\Ascr_g^{b'}}^{\perf}
  \]
  Now the left morphism is pro\'etale and thus weakly \'etale and the right morphism is \'etale by \cite{Hamacher:DeforSpProdStr}~Prop.~4.6. Thus their concatenation is also weakly \'etale.
 
  To ease the notation, we write $\Pscr_\infty' := \Jscr_\infty'^{\perf} \times \Mscr_{\GSp_{2g},\mu'}(b')^{\perf}$ and denote by $a'\colon J_{b'} \times \Pscr_\infty' \to \Pscr_\infty'$ the $J_{b'}$-action constructed above. To prove that $\pi_\infty'$ is a $J_{b'}$-torsor, we have to show that
  \[
   \pr_2 \times a'\colon J_{b'} \times \Pscr_\infty' \to \Pscr_\infty' \times_{\Sscr_G^{b' \perf}} \Pscr_\infty'
  \]
  is an isomorphism. The fact that it is a monomorphism follows from the fact that the action of $J_{b'}$ is faithful. So let $S$ be a perfect scheme and $(P_1,P_2)$ be an $S$-valued point of the right hand side, i.e. $P_1 = (\Acal_1,\lambda_1,\eta_1,j_1;\rho_1)$ and $P_2 = (\Acal_2,\lambda_2,\eta_2,j_2;\rho_2)$ with common image $(\Acal_3, \lambda_3,\eta_3)$ under $\pi_\infty'$. Then there exists a (unique) quasi-isogeny $g \in J_{b'}(S)$ such that the diagram
  \begin{center}
   \begin{tikzcd}
    X_{0,S} \arrow[dashed]{rr}{g} \arrow{d}{j_1} & & X_{0,S} \arrow{d}{j_2} \\
    \Acal_1[p^\infty] \arrow{r}{j_{1 \ast} \rho_1}& \Acal_3[p^\infty] & \Acal_2[p^\infty] \arrow{l}[swap]{j_{2 \ast} \rho_2}
   \end{tikzcd}
  \end{center}
  commutes, i.e.\ such that $(\pr_2 \times a')(g,P_1) = (P_1,P_2)$.
 \end{proof}
 
 \section{Almost product structure in Shimura varieties of Hodge type}
 \label{sect mantovan} 
  The almost product structure of $\Sscr_G^{\bbf}$ will be constructed as a suitable lift of $\pi'_{N,m}$ and $\pi'_\infty$, respectively. Before that, we define the factors of the product. For this we keep the notation of the precious sections. Recall that we fixed a completely slope divisible {\BTT} $(X_0,\lambda_0,t_{0,\alpha})$ inside the the isogeny class corresponding to $\bbf$.
 
 \subsection{Igusa varieties} 
 
 We define the perfect infinite level Igusa variety over $C$ as the locus ${\Jscr_\infty}^{\perf} \subset (\Jscr_\infty' \times_{C'} C)^{\perf}$ where $t_{0,\alpha} = j^\ast t_{G,\alpha}$ for every $\alpha$. In other words, ${\Jscr_\infty}^{\perf}$ parametrises the isomorphisms $(X_0,\lambda_0,t_{0,\alpha}) \cong (\Acal_G,\lambda_G,t_{G,\alpha})$. Assuming for a moment that this gives us a well-defined scheme, we define the Igusa varieties as
 \begin{eqnarray*}
  \Jscr_\infty &:=& \image (\Jscr_\infty^{\perf} \to \Jscr_\infty' \times_{C'} C) \\
  \Jscr_m &:=& \image (\Jscr_\infty^{\perf} \to \Jscr_m' \times_{C'} C).
 \end{eqnarray*} 
 
 \begin{proposition} \label{prop igusa varieties hodge}
  Let $\Gamma_b \subset \Gamma_{b'}$ the topological subgroup of elements stabilising $t_{0,\alpha}$ for every $\alpha$.
  \begin{subenv}
   \item $\Jscr_\infty^\perf$ is a closed union of connected components of $(\Jscr_\infty' \times_{C'} C)^\perf$ .
   \item $\Jscr_\infty \to C$ is profinite pro\'etale with Galois group $\Gamma_b$.
   \item The morphism $r_m\colon J_m \to C$ is finite \'etale with Galois group $\Gamma_{b,m} := \image (\Gamma_b \to \End(X_0[p^m]))$.
  \end{subenv}
 \end{proposition}
 \begin{proof}
 Since the pullback of $\Xcal_0$ is a constant {\BT}, the tensors $t_{0,\alpha}$ and $j^\ast t_{G,\alpha}$ are morphisms of constant $F$-crystals. By \cite[Lemma~3.9]{RapoportRichartz:Gisoc} they are thus locally constant themselves, hence $\Jscr_\infty^\perf$ is a union of connected components. Now by construction, $\Gamma_{b'}$ acts on $\pi_0(\Jscr_\infty' \times_{C'} C)$ and the orbits are in canonical bijection with $\pi_0(C)$. In particular, there are  only finitely many orbits. Within each $\Gamma_{b'}$-orbit, there is a unique $\Gamma_b$-orbit which contains all the connected components contained in $\Jscr_\infty^\perf$. As $\Gamma_b$ is a closed subgroup of $\Gamma_{b'}$, it follows that the set of connected components of $\Jscr_\infty^\perf$ is a closed subspace of $\pi_0(\Jscr_\infty' \times_{C'} C)$. Hence ${\Jscr_\infty}^{\perf} \subset (\Jscr_\infty' \times_{C'} C)^{\perf}$ is a closed union of connected components. The second and third assertion follow from the first and the fact that $\Gamma_b$ acts simply transitively on the fibres of $\Jscr_\infty \to C$.  
 \end{proof}

\subsection{Rapoport-Zink spaces of Hodge type} \label{ss RZ spaces}
  In this section we recall the notion of Rapoport-Zink spaces of Hodge type. Our main reference is \cite{HowardPappas:GSpin} (see also \cite{Kim:RZ}). 

 \begin{definition}[{\cite[Def.~2.3.3]{HowardPappas:GSpin}}]
  Let $({\rm ANilp^{fsm}})$ denote the category of formally smooth formally finitely generated $\breve\ZZ_p$-algebras $A$ such that $p$ is nilpotent in $A$. Define
  \begin{eqnarray*}
   {\rm RZ_{X_0,\lambda_0,t_{0,\alpha}}}\colon {\rm ANilp^{fsm}} & \to & ({\rm Set}) \\
   A &\mapsto& \bigslant{\{(X,\lambda,t_\alpha,\rho)\}}{\cong},
  \end{eqnarray*}
  where $(X,\lambda,t_\alpha)$ is a polarised {\BTT} over $A$ and $\rho\colon X_0 \otimes A/p \to X \otimes A/p$ a quasi-isogeny respecting additional structures such that the following compatibility criteria are met.
  \begin{assertionlist}
   \item The sheaf of isomorphisms
    \[
     \underline\Isom ((\DD(X),\lambda, t_\alpha),(\DD(X_0),\lambda_0,t_{0, \alpha})) 
    \]
    is a crystal of $G_{\breve\ZZ_p}$-torsors w.r.t.\ the natural $G_{\breve\ZZ_p}$-action.
   \item \'Etale locally, we have $\DD(X)_A \cong M_A$ compatible with additional structure such that the image of the Hodge filtration $\Fil^1 \subset \DD(X)_A$ is induced by a cocharacter conjugate to $\mu$.
  \end{assertionlist}
 \end{definition}

 \begin{theorem}[{\cite[Thm.~3.2.1]{HowardPappas:GSpin}}]
  The functor ${\rm RZ}_{X_0,\lambda_0,t_{0,\alpha}}$ is representable by a formally smooth formal scheme $\Mfr_{G,\mu}(\bbf)$ over $\Spf \breve\ZZ_p$, which is formally of finite type.
 \end{theorem}

 The canonical morphism $\Mfr_{G,\mu}(\bbf) \to \Mfr_{\GSp_{2g},\mu'}(\bbf')$ is a closed embedding (\cite[Prop.~3.2.11]{HowardPappas:GSpin}); we identify $\Mfr_{G,\mu}(\bbf)$ with its image. It is stable under the action of $J_b \subset J_{b'}$ by \cite[Thm.~3.2.1]{HowardPappas:GSpin}.  We will work with its underlying reduced subscheme $\Mscr_{G,\mu}(\bbf)$. 
 
 Moreover, $\Mscr_{G,\mu}(\bbf)$ is equipped with a morphism $\Theta_x\colon \Mscr_{G,\mu}(\bbf) \to \Sscr_G^\bbf$ for every $x \in \Jscr_\infty(\FFbar_p)$, which can be described as follows. For $x' \in \Jscr'_\infty(\FFbar_p)$ denote by $x_m'$ its image in $\Jscr'_m(\FFbar_p)$ and let 
 \[
  \Theta'_{x',m} := \pi_{N,m | \{x'_m\} \times \Mscr_{\GSp_{2g},\mu'}(\bbf')^{m_1,m_2}}\colon \Mscr_{\GSp_{2g},\mu'}(\bbf')^{m_1,m_2} \to \Ascr_g^{\bbf'}
 \]
 and define $\Theta'_{x'}:\Mscr_{\GSp_{2g},\mu'}(\bbf') \to \Ascr_g^{\bbf'}$ as their limit (see also \cite[Thm~6.21]{RapoportZink:RZspace}). If $x'$ is the image of a point $x \in \Jscr_\infty(\FFbar_p)$, then $\Theta'_{x' | \Mscr_{G,\mu}(\bbf)}$ factorizes through $\Sscr_G^-$. The morphism $\Theta_x$ is defined as the unique lift of $\Theta'_{x' | \Mscr_{G,\mu}(\bbf)}$ such that $\Theta_x^\ast (\Acal_G[p^\infty],\lambda_G,t_{G, \alpha})$ is the universal {\BTT} over $\Mscr_{G,\mu}(\bbf)$ (see \cite[\S~3.2]{HowardPappas:GSpin}, in particular Lemma~3.2.6 and Remark 3.2.14 for existence, \cite[Prop.~1.3.11]{Kisin:LanglandsRapoport} for uniqueness). Note that the restriction $\Theta_{x,m}$ of $\Theta_x$ to $\Mscr_{G,\mu}(\bbf)^{m_1,m_2} := \Mscr_{G,\mu}(\bbf) \cap \Mscr_{\GSp_{2g},\mu'}^{m_1,m_2}$ depends only on the image of $x$ in $\Jscr_m(\FFbar_p)$. In the following, we will also use the notion $\Theta_{x,m}$ for $x \in \Jscr_m(\FFbar_p)$.

The assumption that $X_0$ is completely slope divisible was not used in the construction of ${\rm RZ_{X_0,\lambda_0,t_{0,\alpha}}}$ above (nor was it used in \cite{HowardPappas:GSpin}). However, note that the isomorphism class of ${\rm RZ_{X_0,\lambda_0,t_{0,\alpha}}}$ does only depend on the choice of $(X_0,\lambda_0,t_{0,\alpha})$ inside its isogeny class; thus we can make this assumption without losing generality. Indeed, if $(X,\lambda,t_\alpha)$ is a {\BTT} over $\FFbar_p$ such that its Dieudonn\'e crystal is isomorphic to $(M_{\breve\ZZ_p},\psi \otimes 1, s_\alpha \otimes 1)$, choose an isogeny $\varphi\colon (X,\lambda,t_\alpha) \to (X_0,\lambda_0,t_{\alpha,0})$ with $X_0$ completely slope divisible. Then we have a canonical isomorphism
 \begin{eqnarray*}
   {\rm RZ}_{X_0,\lambda_0,t_{0,\alpha}} &\isom&  {\rm RZ}_{X,\lambda,t_\alpha}\\ 
  \rho & \mapsto & \rho \circ \varphi,
 \end{eqnarray*}
 in particular we can identify their underlying reduced subschemes. Moreover, for any pair $z = (z_0,i)$ with $z_0 \in \Sscr_G^b(\FFbar_p)$ and $i\colon (X,\lambda, t_\alpha) \isom (\Acal_{G,z_0}[p^\infty],\lambda_{G,z_0},t_{G,\alpha,z_0})$ the characterizing property of $\Theta_z$ implies
 \begin{equation} \label{eq theta}
  \Theta_z = \Theta_x
 \end{equation}
 where $x \in \Jscr_\infty(\FFbar_p)$ given as follows. Its image in $\Sscr_{G,0}$ equals $\Theta_z(\varphi)$ and we choose the isomorphism such that
 \begin{center}
  \begin{tikzcd}
   X \arrow{rr}{\sim}[swap]{i} \arrow[two heads]{d}{\varphi} & & \Acal_{G,z_0}[p^\infty] \arrow[two heads]{d} \\
   X_0 \arrow[dashed]{r}{\sim} & \Acal_{G,\Theta_z(\varphi)}  \arrow{r}{\sim} & \bigslant{\Acal_{G,z_0}}{\ker \varphi}
  \end{tikzcd}
 \end{center}
 commutes.

\subsection{Construction of the correspondence}

 Let $\pi_{N,m}^-\colon \Jscr_m^{(p^{-N})} \times \Mscr_{G,\mu}(\bbf)^{m_1,m_2} \to \Ascr_{g,0}^{\bbf'}$ be the pullback of $\pi_{N,m}'$. We define a lift to $\Sscr_G^\bbf$ on $\FFbar_p$-points by
 \[
  \pi_{N,m}(x,y) = \Theta_{x,m}(y).
 \]
 We already know that the restriction of $\pi_{N,m}$ to $\{x\} \times \Mscr_{G,\mu}(\bbf)^{m_1,m_2}$ is a morphism of varieties. As the next step, we show that $H_{y,N} := \pi_{N,m |\Jscr_m^{(p^{-N})} \times \{y\}}$ is morphism of varieties. For this, we study the possible lifts of a morphism with normal source to the normalisation of the target.
 
 \begin{proposition} \label{prop BC birational}
  Let $\nu\colon \widetilde{S} \to S$ be an integral birational morphism of schemes. Given a quasi-compact morphism $f:T \to S$ with $T$ irreducible, we fix an irreducible component $T_0$ of $\widetilde{T} := T \times_S \widetilde{S}$ which maps surjectively onto $T$. Then the restriction $\nu_0:T_0 \to T$ of the pull-back $\nu_T$ of $\nu$ is again integral and birational. Moreover, we have a bijection
  \begin{eqnarray*}
   \{T_0 \subset \widetilde{T} \textnormal{ irred.\ comp.} \mid \nu(T_0) = T \} &\leftrightarrow& \{T' \subset \widetilde{S} \textnormal{ irreducible, closed} \mid \nu(T') = \overline{f(T)}\} \\
   T_0 &\mapsto& \overline{f_{\widetilde{S}}(T_0)}
  \end{eqnarray*}
 \end{proposition}
 
 \begin{proof}
 As the property of being integral is stable under base change and composition, $\nu_0$ is integral. To show birationality, consider the scheme-theoretic images $W$ and $W_0$ of $T$ in $S$ and $T_0$ in $\widetilde{S}$, respectively. Then $W_0$ is an irreducible component of $W \times_S \widetilde{S}$ which maps surjectively onto $W$. 
 Thus it suffices to show the assertion in the following two cases.
 \begin{assertionlist}
  \item $T \to S$ is a closed immersion
  \item $T \to S$ is dominant.
 \end{assertionlist}
 Then (a) proves that $W_0 \to W$ is integral, birational and substituting $S=W, \widetilde{S}=W_0$ in (b) implies our assertion.
 
 In the first case, assume without loss of generality that $T$ and $S$ are reduced and affine, say $S = \Spec R$, $\widetilde{S} = \Spec{A}$ and $T = \Spec R/\pfr$.  Then $T_0 = \Spec A/\Pfr$ where $\Pfr$ is a prime lying over $\pfr$. As $\nu_0$ is clearly dominant, it suffices to show that the induced  morphism of fraction fields $Q(R/\pfr) \mono Q(A/\Pfr)$ is surjective. Indeed, for any $\frac{r \mod \Pfr}{s \mod \Pfr} \in Q(A/\Pfr)$ write $r = \frac{r_1}{r_2} \mod \Pfr, s = \frac{s_1}{s_2} \mod \Pfr$ with $r_1,r_2,s_1,s_2 \in R$; now
 \[
  \frac{r \mod \Pfr}{s \mod \Pfr} = \frac{r_1/r_2 \mod \Pfr}{s_1/s_2 \mod \Pfr} = \frac{r_1s_2 \mod \Pfr}{s_1r_2 \mod \Pfr} \in Q(R/\pfr).
 \] 
 
 In the second case, the generic fibre of $\widetilde{T}$ over $T$ is the base change of the generic fibre of $\widetilde{S}$ over $S$. As the $\nu$ induces an isomorphism on the generic fibre, so does $\nu_T$. Thus, there exists a unique top-dimensional component $T_0$ of $\widetilde{T}$ and $\nu_0$ is birational.
 \end{proof}
 
 \begin{corollary} \label{cor lift to normalisation}
  Let $f\colon T \to S$ be a quasi-compact morphism of schemes and assume that $T$ is normal. Let $\nu:\widetilde{S} \to S$ be an integral birational morphism. Then we have a one-to-one correspondence.
  \begin{eqnarray*}
   \{\tilde{f}:T \to \widetilde{S} \mid \nu\circ\tilde{f} = f \} &\leftrightarrow& \{(T'_C)_{C \subset T \textnormal{ conn. comp.}} \mid T'_C \subset \widetilde{S} \textnormal{ irred.\ closed}, \nu(T'_C) = \overline{f(C)}\} \\
   \tilde{f} &\mapsto& (\overline{\tilde{f}(C)})
  \end{eqnarray*}
 \end{corollary} 
 \begin{proof}
  Assume without loss of generality that $T$ is irreducible and denote $\widetilde{T}:= T \times_S \widetilde{S}$. We have a sequence of bijections
  \begin{eqnarray*}
    \{\tilde{f}:T \to \widetilde{S} \mid \nu\circ\tilde{f} = f \}  &\leftrightarrow& \{s\colon T \to \widetilde{T} \mid \nu_T \circ s = \id \} \\
    &\leftrightarrow& \{T_0 \subset T \textnormal{ irreducible component} \mid \nu_T(T_0) =T \} \\
    &\leftrightarrow& \{T' \subset \widetilde S \textnormal{ irreducible,\ closed} \mid \nu(T') = \overline{f(T)} \}.
  \end{eqnarray*}
  The first bijection is given by the universal property of the fibre product $T \times_S \widetilde{S}$. For the second bijection note that for every  $T_0 \subset \widetilde{T}$ of the right hand side, the restriction $T_0 \to T$ is integral and birational by the previous proposition and hence an isomorphism as $T$ is normal. In particular there exists a unique section $T \to T_0$. The third bijection is the last assertion of the previous proposition.
 \end{proof} 
 
 As $\Jscr_m$ is smooth and thus in particular normal, we can apply the last corollary to construct a lift $H_{y,N,(x_J)}$ of $\pi^-_{N,m|\Jscr_m^{(p^{-N})} \times \{y\}}$ as follows: For each connected component $J \subset \Jscr_m^{(p^{-N})}$ we fix a closed point $x_J \in J$. Let $x_J^+ := \pi_{N,m}(x_J,y)$ and $x_J^- := \nu_{\rm Sh} (x_J^+)$. Using the identification of Proposition~\ref{prop local-global},  $\pi_{N,m}^-$ maps $\Jscr_{x_J}^\wedge \times \{y\}$ onto the central leaf of $\Def(\Acal^\univ[p^\infty]_{x_J^-},\lambda^\univ_{x_J^-}, t_{G,\alpha,x_J^+})$ by \cite[Prop.~4.6]{Hamacher:DeforSpProdStr}. Thus $\pi_{N,m}^-(J \times \{y\})$ is its closure in $\Sscr_{G,0}^-$, which is the image of the closure $C(x_J)$ of the central leaf in $\Def(\Acal_G[p^\infty]_{x_J^+},\lambda_{G,x_J^+}, t_{G,\alpha,x_J^+})$ in $\Sscr_{G,0}$. Note that $C(x_J)$ is the connected component of a central leaf containing $x_J^+$. By Corollary~\ref{cor lift to normalisation} the family $(C(x_J))_{J \subset \Jscr_m \textnormal{ irr.\ comp.}}$ yields a lift $H_{y,N,(x_J)}$ of $\pi^-_{N,m|\Jscr_m^{(p^{-N})} \times \{y\}}$. 
 
 \begin{proposition} \label{prop H is a morphism}
  Let $\rho_y$ be the quasi-isogeny corresponding to $y \in \Mscr_{G,\mu}(\bbf)(\FFbar_p)$ and denote by $\rho_{Ig,y}$ the quasi-isogeny given by
  \[
   p^{m_1}\rho_{Ig,y}\colon \Acal_{G,\Jscr_M^{(p^{-N})}}[p^\infty] \epi \bigslant{\Acal_{G,\Jscr_M^{(p^{-N})}}[p^\infty]}{j_m(\ker (p^{m_1}\rho_y))} = H_{y,N,(x_J)}^\ast \Acal_G[p^\infty].
  \]
  \begin{subenv}
   \item The family of crystalline Tate tensors $\rho_{Ig,y}^{\ast}t_{G,\alpha}$ coincide with the natural crystalline Tate tensors on $\Acal_{G,\Jscr_m^{(p^N)}}$ .  
   \item $H_{y,N,(x_J)} = H_{y,N}$, in particular the lift is independent of the choice of the points $x_J$.
  \end{subenv}
 \end{proposition}  
 \begin{proof}
  By the rigidity of crystals it suffices to show the equality of the crystalline Tate tensors in one point of each connected component. Indeed, they are equal over $x_J$ by construction. The second part follows directly from \cite[Prop.~1.3.11]{Kisin:LanglandsRapoport}, where Kisin shows that the crystalline Tate tensors determine a point in a fibre of $\nu_{\rm Sh}$ uniquely.
 \end{proof}
 
 So far we have shown that the restriction of $\pi_{N,m}$ to subvarieties of the form $\Jscr_{(p^{-N})} \times \{y\}$ and $\{x\} \times \Mscr_{G,\mu}(\bbf)$ is a morphism of varieties. We now deduce the general result.  

 \begin{theorem}\label{thm almost product}
  $\pi_{N,m}$ is a morphism of varieties.
 \end{theorem} 
 \begin{proof}
 
  Denote $\widetilde{T} := (\Jscr_m^{(p^{-N})} \times \Mscr_{G,\mu}(\bbf)^{m_1,m_2}) \times_{\Sscr_G^-} \Sscr_G$ and let $\nu_T\colon \widetilde{T} \to \Jscr_m^{(p^{-N})} \times \Mscr_{G,\mu}(\bbf)^{m_1,m_2}$ be the canonical projection. Via the universal properties of the  fibre product, $\pi_{N,m}$ induces a (set-theoretical) section $s\colon (\Jscr_m^{(p^{-N})} \times \Mscr_{G,\mu}(\bbf)^{m_1,m_2})(\FFbar_p) \to \widetilde{T}(\FFbar_p)$. We have to show that $s$ is a morphism of varieties.
  
  For this it suffices to show that $\nu_T$ is locally an isomorphism at $s(x,y)$ for every $(x,y) \in  (\Jscr_m^{(p^{-N})} \times \Mscr_{G,\mu}(\bbf)^{m_1,m_2})(\FFbar_p)$. Indeed, in this case it there exists a unique section (in the set- and scheme-theoretic sense) in a neighbourhood $U$ of $(x,y)$ mapping $(x,y) \mapsto s(x,y)$. By uniqueness this section must be given by $s_{|U}$. As these neighbourhoods cover $\Jscr_m^{(p^{-N})} \times \Mscr_{G,\mu}(\bbf)^{m_1,m_2}$, the claim follows.
  
  First, we show that $\nu_T$ induces an isomorphism of formal neighbourhoods at $s(x,y)$, i.e. it is \'etale at $s(x,y)$.
  For this it suffices to show that the restriction $\pi_{N,m}^-$ to the formal neighbourhood of $(x,y)$ in $\Jscr_m^{(p^{-N})} \times \Mscr_{G,\mu}(\bbf)^{m_1,m_2}$ factorizes through $\Sscr_{G\, \pi_{N,m}(x,y)}^\wedge \subset \Sscr_{G\, \pi_{N,m}^- (x,y)}^{-\, \wedge}$. Since the formal neighbourhood in a Igusa variety is identified with the central leaf in the deformation space under the isomorphism of Proposition~\ref{prop local-global}, this follows from \cite[Prop.~4.5]{Hamacher:DeforSpProdStr}.
  
  Thus it suffices to show the restriction of $\nu_T$ to the stalk at $s(x,y)$ induces a birational morphism, as any finite \'etale  birational morphism is an isomorphism. By Proposition~\ref{prop BC birational} the restriction of $\nu_T$ to irreducible components is birational, so it suffices to show that $\nu_T$ induces a bijection of irreducible components of the stalks. This is obvious over the smooth locus as in this case the irreducible components of the stalk are in canonical one-to-one correspondence with the irreducible components of the formal neighbourhood. For the general case we need the following result.  
  
  Let $Z_1 \subset \Jscr_m^{(p^{-N})}$, $Z_2 \subset \Mscr_{G,\mu}(\bbf)^{m_1,m_2}$ be irreducible components. Then there exists a unique irreducible component of $\widetilde{T}$ containing $s(Z_1(\FFbar_p) \times Z_2(\FFbar_p))$. Indeed, let $z_1 \in Z_1(\FFbar_p)$ and $z_2 \in Z_2(\FFbar_p)$ be smooth points. As $s$ defines an isomorphism of varieties in a neighbourhood of $(z_1,z_2)$, its image $s(z_1,z_2)$ is also a smooth point and thus contained in a unique irreducible component $Z \subset \widetilde{T}$. Since the restriction $s_{|Z_1(\FFbar_p) \times \{z_2\}}$ is a morphism of varieties we have $s(Z_1(\FFbar_p) \times \{z_2\}) \subset Z(\FFbar_p)$. Now $s(Z_1(\FFbar_p) \times \{z_2\})$ lies in the smooth locus, as $Z_1$ is smooth. Thus we can apply the same argument again, this time varying the second coordinate, and obtain $s(Z_1(\FFbar_p) \times Z_2(\FFbar_p)) \subset Z(\FFbar_p)$.
  
  Denote by $\Irr(\cdot)$ the set of irreducible components of a scheme. Consider the diagram
  \begin{center}
  \begin{tikzcd}
   \Irr(\widetilde{T}_{s(x,y)}^\wedge) \arrow[two heads]{r} \arrow{d}{\rotatebox[origin=c]{-90}{$\sim$}} & \Irr(\widetilde{T}_{s(x,y)}) \arrow[bend right, two heads]{d} \\ \Irr((\Jscr_m^{(p^{-N})} \times \Mscr_{G,\mu}(b)^{m_1,m_2})_{(x,y)})^\wedge) \arrow[two heads]{r} & \Irr((\Jscr_m^{(p^{-N})} \times \Mscr_{G,\mu}(b)^{m_1,m_2})_{(x,y)}) \arrow[bend right]{u}
  \end{tikzcd}
  \end{center}
  where the right most arrow is the section which maps $(Z_1,Z_2)$ to $Z$ in the above construction. As both squares commute, one sees easily that the right arrows must be bijections.
 \end{proof}

 The geometric properties of $\pi_{N,m}$ follow readily from the analogue properties of $\pi'_{N,m}$.

 \begin{proposition}
  $\pi_{N,m}$ is finite, and surjective for $m$ big enough.
 \end{proposition}
 \begin{proof}
  Let $z \in \Sscr_G(\FFbar_p)$ be arbitrary. We choose an isogeny $\varphi\colon (\Acal_{G,z},\lambda_{G,z},t_{G,\alpha,z}) \to (X_0,\lambda_0,t_{0,\alpha}) $; by \cite[Lemma~4.4]{Scholze:LKModCurve} we can
 choose $\varphi$ such that $\varphi^{-1} \in \Mscr_{G,\mu}(\bbf)^{0,m}$ where $m$ does not depend on $z$. By (\ref{eq theta}) we have
 \[
  z = \Theta_x(\varphi^{-1})
 \]
 for some $x \in \Jscr_{m}$. The finiteness from Proposition~\ref{prop geometric properties of mantovans morphism}. As $\pi_{N,m}'$ is finite, so is $ \pi_{N,m}^-$.  Thus $\pi_{N,m}$ is finite as finiteness is satisfies the cancellation property (see for example \cite[Tag~035D, (a)]{AlgStackProj}).
 \end{proof}

 \begin{proposition} \label{prop compatibility of almost product structure}
  \begin{subenv}
   \item $\pi_{N,m} \circ (F_{\Jscr_m/k} \times \id) = \pi_{N+1,m}$
   \item Let $r_{m+1,m}\colon \Jscr_{m+1} \epi \Jscr_m$ be the canonical projection. Then $\pi_{N,m-1} \circ (r_{m+1,m} \times \id) = \pi_{N,m}$, where its is assumed that $m_1+m_2 \leq m-1$.
   \item Let $\iota\colon \Mscr_{G,\mu}(\bbf)^{m_1',m_2'} \mono \Mscr_{G,\mu}(\bbf)^{m_1,m_2}$ the obvious embedding for $m_1' \leq m_1, m_2' \leq m_2$. Then $\pi_{N,m,m_1,m_2} \circ \iota = \pi_{N,m',m_1',m_2'}$.
  \end{subenv}
 \end{proposition}
 \begin{proof}
  It suffices to check these properties on $\FFbar_p$ points. Here they are an easy consequence of the equality $\pi_{N,m}(x,y) = \Theta_x(y)$.
 \end{proof}
 
 \subsection{The almost product structure of infinite level}

 Let $J_b \subset J_b'$ the topological subgroup of elements which stabilise the family $(t_{0,\alpha})$. We view $J_b$ as locally profinite group scheme over $\FFbar_p$.

 \begin{proposition} \label{prop comparision of Igusa varieties}
  The canonical morphism $\Jscr_\infty^\perf \to  \Jscr_\infty'^\perf$ is a closed embedding with $J_b$-stable image. More explicitly, for any $g \in J_b$, $x=(x_0;j) \in \Jscr_\infty(\FFbar_p)$ the $J_b$ action is given by
  \begin{equation} \label{eq J_b action}
   g.x = (\Phi_x(g^{-1});(g^{-1})_\ast j).
  \end{equation}
 \end{proposition}
 \begin{proof}
  Let $x' = (\Acal,\lambda,\eta;j) \in \Jscr'_\infty(\FFbar_p)$, then any lift $x \in \Jscr_\infty(\FFbar_p)$ satisfies $t_{G,\alpha,x} = j_\ast t_{0,\alpha}$. Hence any such lift, if it exists, is unique by \cite[Cor.~1.3.11]{Kisin:LanglandsRapoport}. Thus $\Jscr_\infty \to \Jscr_\infty'$ is injective on closed geometric points and hence universally injective as the Igusa varieties are Jacobson. Since it is finite, $\Jscr_\infty \to \Jscr_\infty'$ is a universal homeomorphism onto its (closed) image. As a universal homeomorphism of perfect schemes is an isomorphism by \cite[Lemma~3.8]{BhattScholze:AffGr}, $\Jscr_\infty^\perf \to \Jscr_\infty'^\perf$ is a closed embedding. The second assertion follows from the fact that for any $g \in J_b'$, $x = (x_0;j) \in \Jscr_\infty(\FFbar_p)$ the $J_b'$-action is given by
  \[
   g.x = (\Phi'_x(g^{-1});(g^{-1})_\ast j),
  \]
  which is obviously lifted by (\ref{eq J_b action}).
 \end{proof}

 Let $\pi_\infty\colon \Jscr_\infty^{\perf} \times \Mscr_{G,\mu}(\bbf) \to \Sscr_{G,0}^\bbf$ be the limit of $\pi_{N,m}$ for $N,m \to \infty$. By Proposition \ref{prop compatibility of almost product structure} this is well-defined.

 \begin{proposition} \label{prop infinite almost product}
  The morphism $\pi_\infty^\perf$ is a $J_b$-torsor for the pro\'etale topology
 \end{proposition}
 \begin{proof}
  Recall that for any $x \in \Jscr_m(\FFbar_p)$ and any interior point $y\in \Mscr_{G,\mu}(\bbf)^{m_1,m_2}$ the isomorphism in Proposition~\ref{prop local-global} identifies the restriction of $\pi_{N,m}$ to the formal neighbourhood at $(x,y)$ with the almost product structure considered in \cite{Hamacher:DeforSpProdStr}. Thus it follows that $\pi_\infty^\perf$ is weakly \'etale by the same argument as in Proposition~\ref{prop almost product siegel infinite level} for $\pi_\infty'$. Since the source of $\pi_\infty$ is Jacobson, it suffices to check its $J_b$-invariance on $\FFbar_p$-points. For $(x,y) \in (\Jscr_\infty \times \Mscr_{b,\mu}(\bbf))(\FFbar_p)$ and $g \in J_b$ one has
  \[
   \pi_\infty(g.(x,y)) = \Theta_{g.x}(g.y) \stackrel{(\ref{eq theta})}{=} \Theta_x(y) = \pi_\infty(x,y).
  \]
  Here (\ref{eq theta}) is applied to $(X,\lambda,t_\alpha) = (X_0,\lambda_0,t_{0,\alpha})$ and $\phi=g$.

  To ease the notation, we write $\Pscr_\infty := (\Jscr_\infty \times \Mscr_{G,\mu}(\bbf))^\perf$ and denote by $a\colon J_b \times \Pscr_\infty \to \Pscr_\infty$ the $J_b$-action. We have to show that
  \[
   \pr_2 \times a\colon J_b \times \Pscr_\infty \to \Pscr_\infty \times_{\Sscr_{G,0}^{\bbf\, \perf}} \Pscr_\infty
  \]
  is an isomorphism. Since the target is reduced and Jacobson, it suffices to show that
  \begin{assertionlist}
   \item $\pr_2 \times a$ is a closed embedding and
   \item $\pr_2 \times a$ is surjective on $\FFbar_p$-points.
  \end{assertionlist}
  The first assertion follows directly from the commutative diagram
  \begin{center}
   \begin{tikzcd}
    J_b \times \Pscr_\infty \arrow{r}{\pr_2 \times a} \arrow[hook]{d}[swap]{\textnormal{closed} \atop \textnormal{embedding}} & \Pscr_\infty \times_{\Sscr_{G,0}^\perf} \Pscr_\infty \arrow{d} \\
    J_{b'} \times \Pscr'_\infty \arrow{r}{\sim}[swap]{\pr_2 \times a'} & \Pscr'_\infty \times_{\Ascr_g^\perf} \Pscr'_\infty.
   \end{tikzcd}
  \end{center}
  The second assertion is equivalent to the claim that any two points $(x,y), (u,w) \in \Pscr_\infty(\FFbar_p)$ with the same image under $\pi_\infty$ lie in the same $J_b$ orbit. Indeed $\pi_\infty(x,y) = \pi_\infty(u,w)$ implies $\pi_\infty'(x,y) = \pi_\infty'(u,w)$, thus there exists a $g \in J_{b'}$ such that $g.(u,w) = (x,y)$. More precisely, $g$ is the concatenation $X_0 \stackrel{y}{\longrightarrow} \Acal_{G,\pi_\infty(x,y)}[p^\infty] = \Acal_{G,\pi_\infty(u,w)}[p^\infty] \stackrel{w^{-1}}{\longrightarrow} X_0$. Thus $g$ fixes $t_{0,\alpha}$ for every $\alpha$, i.e.\ $g\in J_b$.
 \end{proof}

 \section{Relation with the Caraiani-Scholze product structure}
\label{sect caraiani scholze}
 \subsection{Construction of the product structure in the Siegel moduli space}
 We briefly sketch the construction of the product structure of Caraiani and Scholze for the Siegel moduli space as presented in \cite[\S~4]{CaraianiScholze:ShVar}.
 
 We fix a lift $(\Xtilde_0,\tilde\lambda_0)$ of $(X_0,\lambda_0)$ over $O_E$. Since $\Jscr_\infty'^\perf$ is a perfect scheme, there exists it can be uniquely extended to a flat formal scheme $\Igfr^{\bbf'}$ over $\Spf \breve\ZZ_p$; more explicitely, its structure sheaf is given by $U \mapsto W(\Ocal_{\Jscr_\infty'^\perf}(U))$. We define the extension $\rtilde_\infty'\colon \Igfr^{\bbf'} \to \Ascr_g$ of $r_\infty$ as the morphism parametrising isomorphisms $(\Acal^\univ[p^\infty],\lambda^\univ) \cong (\Xtilde_0,\tilde\lambda_0)$ (\cite[Lemma~4.3.10]{CaraianiScholze:ShVar}). Moreover, we define the set-valued functor $\Xfr^{\bbf'}$ which maps an $O_E$-algebra $R$ with $p \in R$ nilpotent to
 \[
  \Xfr^{\bbf'}(R) := \{(A,\lambda,\eta; \rho) \mid (A,\lambda,\eta) \in \Ascr_g(R), \rho\colon (A[p^\infty],\lambda) \otimes R/p \to (X_0,\lambda_0) \otimes R/p \textnormal{ quasi-isogeny} \}.
 \] 
 Now consider the morphism $\Igfr^{\bbf'} \times \Mfr_{\GSp_{2g},\mu'} \to \Xfr^{\bbf'}$ mapping $(\Acal,\lambda,\eta;j),(\Xcal,\rho) \in (\Igfr^{\bbf'} \times \Mfr_{G,\mu}(\bbf'))(R)$ to $(\Acal',\lambda',\eta';\rho)$ which is constructed as follows. We lift $\rho$ to $R$ and choose $m$ big enough such that $p^m\rho$ is an isogeny. We endow
 \[
  \Acal' := \bigslant{\Acal}{j(\ker(p^m\rho))}
 \] 
  with the polarisation $\lambda'$ and $K'^p$-level structure $\eta'$ induced by $\lambda$ and $\eta$, respectively. We denote by $\rho'\colon X_0 \to \Acal'$ the quasi-isogeny such that the diagram
 \begin{center}
  \begin{tikzcd}
   X_0 \arrow{r}{j} \arrow[bend left]{rr}{p^m \rho'} & \Acal[p^\infty] \arrow[two heads]{r} & \Acal'[p^\infty]
  \end{tikzcd}
 \end{center}
 commutes.

 \begin{proposition}[{\cite[Lemma~4.3.12]{CaraianiScholze:ShVar}}]
  The above morphism $\Igfr^{\bbf'} \times \Mfr_{\GSp_{2g},\mu'}(\bbf') \to \Xfr^{\bbf'}$ is an isomorphism.
 \end{proposition}
 
 Let $\tilde\pi'_\infty$ denote the concatenation $\Igfr^{\bbf'} \times \Mfr_{\GSp_{2g},\mu'} \isom \Xfr^{\bbf'} \stackrel{can.}{\longrightarrow} \Ascr_g$.
  
 \subsection{The local product structure}
 Since the $\mfr$-adic topology on the stalks of $\Igfr^{\bbf'} \times \Mfr_{\GSp_{2g},\mu'}(\bbf')$ is not separable, considering their completions would be unrewarding. Instead, we consider the following construction.
 
 \begin{definition}
  Let $A$ be an $R$-algebra where $R$ is Noetherian and $I \subset A$ be an ideal. The restricted $I$-adic completion of $A$ relative to $R$ is defined as
  \[
   \Ahat = \underrightarrow\lim\, \hat{A}_\lambda
  \]
  where $A = \underrightarrow\lim A_\lambda$ is a filtered colimit with $A_\lambda$ finitely presented over $R$ and $\hat{A}_\lambda$ denotes their completion with respect to the preimage $I_\lambda$ of $I$ in $A_\lambda$.
 \end{definition}
 
 The restricted completion has the advantage that it has similar properties as the completion of Noetherian rings. 
 
 \begin{lemma} \label{lemma restricted completion}
  Let $R$ be a Noetherian ring, $A$ an $R$-algebra and $I$ an ideal of $A$.
  \begin{subenv}
   \item The restricted completion $\hat{A}$ does not depend on the choice of the $A_\lambda$. In particular, if $A$ is of finite presentation over $R$, $\hat{A}$ is just the $I$-adic completion of $A$.
   \item $A \rightarrow \hat{A}$ is flat.
   \item Let $B$ be a finite $A$-algebra, $J = IB$ and $\Bhat$ the restricted completion of $B$ relative to $R$. Then $\Bhat = B \otimes_A \Ahat $.
   \item The image of $\Spec{\hat{A}}$ in $\Spec{A}$ is the completion of $V(I)$, i.e.\ the pro-open subset of all points which specialise to a point in $V(I)$.
   \item If $A$ is local with maximal ideal $I$, then for any ideal $J \subset A$ we have $ \Ahat\cdot J  \cap A = J$.
  \end{subenv}
 \end{lemma}
 \begin{proof}
  Assume that $A = \underrightarrow\lim A_\lambda = \underrightarrow\lim B_\mu$ where $A_\lambda$ and $B_\mu$ are finitely presented $R$-algebras. Thus for any $\lambda$ there exists a $\mu(\lambda)$ such that the canonical morphism $A_\lambda \to A$ factorises over $B_{\mu(\lambda)}$ and for any $\mu$ there exists a $\lambda(\mu)$ such that $B_\mu \to A$ factorises over $A_{\lambda(\mu)}$. Since the morphisms are continuous, they induce a morphisms $\phi_\lambda\colon \hat{A}_\lambda \to \hat{B}_\mu$ and $\psi_\mu\colon \hat{B}_\mu \to \hat{A}_\lambda$. Since the diagrams
  \begin{center}
   \begin{tikzcd}
    \Ahat_{\lambda(\mu(\lambda))} & & & \Bhat_{\mu(\lambda(\mu))} \\
   & \Bhat_{\mu(\lambda)} \arrow{ul}{\psi_{\mu(\lambda)}} & \Ahat_{\lambda(\mu)} \arrow{ur}{\phi_{\lambda(\mu)}} & \\
    \Ahat_\lambda \arrow{uu}{can.} \arrow{ur}{\phi_\lambda} & & & \Bhat_\mu \arrow{ul}{\psi_\mu} \arrow{uu}{can.}
   \end{tikzcd}
  \end{center}
  commute, taking the limit of $(\phi_\lambda)$ and $(\psi_\mu)$ yields two mutually inverse isomorphisms $\underrightarrow\lim \Ahat_\lambda \cong \underrightarrow\lim \Bhat_\mu$. The second assertion follows from $\hat{A} = \underrightarrow\lim \hat{A}_\lambda = \underrightarrow\lim A \otimes_{A_\lambda} \hat{A}_\lambda$, as the limit of flat $A$-modules is flat. To prove the third assertion, write $A = \underrightarrow\lim A_\lambda$, $B=\underrightarrow\lim B_\lambda$ such that each $B_\lambda$ is a finite $A_\lambda$-algebra. Then
  \[
   \Bhat = \underrightarrow\lim B \otimes_{B_\lambda} \Bhat_\lambda = \underrightarrow\lim B \otimes_{A_\lambda} \Ahat_\lambda = B \otimes \Ahat.
  \] 
  The forth assertion is a direct consequence of the analogous assertion for $A_\lambda \to \Ahat_\lambda$. To see the last assertion, let $A'_\lambda$ be the localisation of $A_\lambda$ at $I \cap A_\lambda$. Note that $A$ and $\Ahat$ are the limits of $A_\lambda'$ and $\Ahat_\lambda'$, respectively. Now Krull's intersection theorem gives
  \[
   \Ahat_\lambda' \cdot (J \cap A_\lambda) = J \cap A_\lambda,
  \]
  which yields $\Ahat \cdot J \cap A = J$ after taking the limit.
 \end{proof}
 
 We deduce the following analogue of \cite[\S~III.5, Thm.3]{Mumford:RedBook}.

 \begin{corollary} \label{cor criterion for etale}
  Let $\phi:(A,\mfr) \to (B,\nfr)$ be a local homomorphism of finite type of $R$-algebras such that the induced field extension of residue fields $\overline\phi\colon A/\mfr \to B/\nfr$ and the induced morphism $\hat\phi\colon \Ahat \to \Bhat$ are isomorphisms. Then $\phi$ is \'etale.
 \end{corollary}
 \begin{proof}
  Since $\overline\phi$ is an isomorphism, we have $\Omega_{B/A} \otimes A/\mfr = 0$ by part (5) of above lemma, hence $\phi$ is formally unramified (and thus unramified) by Nakayama's lemma. Since $\hat\phi$ is an isomorphism, parts (2) and (4) of the above lemma moreover imply that $\phi$ is flat and the claim follows.
 \end{proof}

 \begin{definition}
  We define the restricted formal neighbourhood of a formal scheme $\Xfr$ over $\Spf \breve\ZZ_p$ at a point $x$ as the restricted $\mfr_{\Xfr,x}$-adic completion of $\Ocal_{\Xfr,x}$ relative to $\breve\ZZ_p$.
 \end{definition} 
 
 Let $\xtilde_0 \in \Jscr_\infty(k)$, with images  $x_0 \in C(k), x_0' \in C'(k)$ and $x \in \Mscr_{G,\mu}(\bbf)(k)$ with corresponding quasi-isogeny $\rho\colon (X_0,\lambda_0,t_{0,\alpha}) \to (X,\lambda,t_\alpha)$. Denote by $(X_0^\loc, \lambda_0^\loc)$ and $(X^\loc,\lambda^\loc)$ the universal deformations of $(X_0,\lambda_0)$ and $(X,\lambda)$, respectively.
  The canonical isomorphism $\Ascr_{g,x_0'}^\wedge \isom \Def(X_0,\lambda_0)$ identifies the fibre of $\rtilde_\infty'$ over the formal neighbourhood with the formal scheme ${\Cfr_\infty'}^{\Gamma_{b'}}$ parametrising isomorphisms $j\colon \Xtilde_0 \isom X_0^{\loc}$.
   The restricted formal neighbourhood of $\xtilde_0$ in $\Igfr^{\bbf'}$ is thus given by a connected component $\Cfr_\infty' \subset {\Cfr_\infty'}^{J_{b'}}$. Hence the restriction of $\tilde \pi_\infty'$ to the restricted formal neighbourhood at the closed points corresponds to a morphism
 \[
  \tilde\pi_{\infty,\loc}'\colon \Cfr_\infty' \times \Mfr_{\GSp_{2g},\mu'}(\bbf')_x^\wedge \to \Def(X,\lambda).
 \]
 
 \begin{remark}
  Denote by $\rho^\loc$ the restriction of the universal quasi-isogeny over $\Mfr_{\GSp_{2g},\mu'}$. Then we have 
 \begin{equation} \label{eq local product structure}
  \tilde\pi_{\infty,\loc}'^\ast X^\loc = \bigslant{ X_0 }{ \ker \rho^\loc}.
 \end{equation}
 One cannot use this relation directly to define $\tilde\pi_{\infty,\loc}'$ via the moduli description of $\Def(X,\lambda)$ since its source is not formally of finite type. However, one can define it as the limit of the following morphisms on ``finite level''. Let $\Cfr_m'$ be the scheme theoretic image of $\Cfr_\infty$ in the formal scheme parametrising isomorphisms $j_m:\Xtilde_0[p^m] \isom X_0^\loc[p^m]$. Note that $\Cfr_m'$ is finite over $\Def(X_0,\lambda_0)$. Let $\Ifr_m' \subset \Mfr_{\GSp_{2g},\mu'}(\bbf')^\wedge_x$ be the locus where there exist $m_1,m_2$ with $m_1 + m_2 \leq m$ such that $p^{m_1}\rho^\loc$ and $p^{m_2}\rho^\loc$ are isogenies. Then $X_{0, \Cfr_m'}^\loc/j_m(\ker p^{m_1} \rho^\loc)$ is a deformation of $X$ and thus induces a morphism
\[
 \tilde\pi_{m,\loc}'\colon \Cfr_m' \hat\times \Ifr_m' \to \Def(X,\lambda).
\]   
 \end{remark} 
 
 Finally, we can construct the local version of the product structure for Shimura varieties of Hodge type. Let $\Igfr^{\bbf}$ be the (unique) flat lift of $\Jscr_b^\perf$ over $\Spf \breve\ZZ_p$. In terms of the identification of Proposition~\ref{prop local-global}, the restricted formal neighbourhood of $\Igfr^{\bbf}$ at $\xtilde_0$ corresponds to the preimage $\Cfr_\infty \subset \Cfr_\infty'$ of $\Def(X_0,\lambda_0,t_{0,\alpha})$

 \begin{lemma} \label{lem local product structure}
  The restriction $\pi_{\infty,\loc}$ of $\pi_{\infty,\loc}'$ to $\Cfr_\infty \times \Mfr_{G,\mu}(\bbf)_x^\wedge$ factorises through $\Def(X,\lambda,t_\alpha)$.
 \end{lemma}  
 \begin{proof}
  First note that because of the formal smoothness of the Rapoport-Zink space and the central leaf we have an abstract isomorphism \[\Cfr_\infty \times \Mscr_{G,\mu}(\bbf)_x^\wedge \cong \Spf \breve\ZZ_p\pot{x_1^{1/p^\infty},\ldots,x_d^{1/p^\infty},y_1,\ldots,y_n}\] for some $d,n$. Since $\Def(X,\lambda)$ is formally of finite presentation $\tilde\pi_{\infty,\loc}$ factorises through $\Spf \breve\ZZ_p\pot{x_1^{1/p^N},\ldots,x_d^{1/p^N},y_1,\ldots y_n}$ for some $N$. By Faltings' criterion (see e.g.\ \cite[Thm.~3.6]{Kim:RZ}) it suffices to show that (\ref{eq local product structure}) respects crystalline Tate tensors. This is a direct consequence of (a slight generalisation of) \cite[Prop.~2.16]{Hamacher:DeforSpProdStr}.
 \end{proof}

 \subsection{The product structure over Shimura varieties of Hodge type}
 Since the formal schemes we consider are not formally of finite type, we need to consider a larger class of test objects than the formally smooth, formally finitely generated $\breve\ZZ_p$-algebras used in \cite{HowardPappas:GSpin}.

 \begin{definition}
  Let $m \in \NN \cup \{\infty\}$ and $A$ be a $p$-adically complete $\breve\ZZ_p/p^m$-algebra.
  \begin{subenv}
   \item $A$ is called perfectly formally finitely generated if there exist positive integers $r_1,r,s_1,s$ such that $A$ can be written as a quotient of
   \[
    \bigslant{\breve\ZZ_p}{p^m}\pot{x_1,\ldots,x_{r_1},x_{r_1+1}^{1/p^\infty},\ldots,x_r^{1/p^\infty}}[y_1,\ldots,y_{s_1}, y_{s_1+1}^{1/p^\infty},\ldots,y_s^{1/p^\infty}]
   \]
   \item $A$ is called essentially perfectly formally finitely generated if it is the $p$-adic completion of an  ind\'etale algebra over some perfectly formally finitely generated $\breve\ZZ_p/p^m$ algebra.
   \item $A$ is called essentially smooth if it is essentially perfectly formally finitely generated and its restricted completion at any maximal ideal is isomorphic to the $p$-adic completion of $\bigslant{\breve\ZZ_p}{p^m}\pot{x_1,\ldots,x_{r_1},x_{r_1+1}^{1/p^\infty},\ldots,x_r^{1/p^\infty}}$ for some $r,r_1$.
   \item Denote by ${\rm Nilp}_{\breve\ZZ_p}^{es}$ the full category of $\breve\ZZ_p$-modules which are $p^m$-torsion and essentially smooth as $\breve\ZZ_p/p^m$-algebras for some $m\in \NN$.
  \end{subenv}
 \end{definition}

 \begin{example} \label{example essentially smooth}
  \begin{subenv}
   \item If $A$ is a formally finitely generated $\breve\ZZ_p/p^m$-algebra, it is essentially smooth if and only if it is formally smooth.
   \item If $A$ is a formally finitely generated smooth $\FFbar_p$-algebra, then $A^\perf$ is essentially smooth.
  \end{subenv}
 \end{example}

  We extend the notion of essential smoothness to formal schemes.

 \begin{definition}
  A formal scheme $\Xfr$ over $\Spf W$ is called essentially smooth (resp.\ locally perfectly formally of finite type, locally essentially perfectly formally of finite type) if it is Zariski locally of the form $\Spf A$ where $A$ is an essentially smooth $\breve\ZZ_p$-algebra (resp.\ perfectly formally of finite type, essentially perfectly formally of finite type).
 \end{definition}

 Using the usual Yoneda argument, one checks that any essentially smooth formal $\breve\ZZ_p$-scheme is uniquely determined by its functor of points restricted to the full subcategory of affine formal schemes whose ring of global sections lies in ${\rm Nilp}_{\breve\ZZ_p}^{es}$.

 \begin{lemma}
  The formal schemes $\Mfr_{G,\mu}(\bbf)$ and $\Igfr^\bbf$ are essentially smooth.
 \end{lemma}
 \begin{proof}
  Since $\Mfr_{G,\mu}(\bbf)$ is formally smooth and locally formally of finite type by \cite[Prop.~3.2.7]{HowardPappas:GSpin}, it follows from the example above that it is essentially smooth. To see that $\Igfr^\bbf$ is essentially smooth denote by $\Igfr_m$ the (unique) flat lift of $\Jscr_m^\perf$ over $\Spf W$. Then $\Igfr_m$ is perfectly formally of finite type and moreover the lift $\Igfr_{m+1} \to \Igfr_m$ is \'etale since it is the unique lift to a morphism of flat formal schemes over $\Spf\breve\ZZ_p$ and the existence of an \'etale lift is guaranteed by \cite[04DZ]{AlgStackProj}. Thus $\Igfr^\bbf$ is essentially perfectly formally of finite type. By construction the restricted local neighbourhood $x$ at a closed point of $\Igfr^\bbf$ equals the limit of restricted formal neighbourhoods of the image of $x$ in $\Igfr_m$. As a consequence of Corollary~\ref{cor criterion for etale} all transition morphisms of this limit are isomorphisms; thus it suffices to check that $\Igfr_0$ is essentially smooth. Since the formal neighbourhood of $C = \Jscr_0$ at a closed point is isomorphic to $\FFbar_p\pot{x_1,\ldots,x_{2\langle \rho, \nu \rangle}}$, the restricted formal neighbourhood of its perfection equals $\FFbar_p\pot{x_1^{1/p^\infty},\ldots,x_{2\langle \rho, \nu \rangle}^{1/p^\infty}}$. Thus the restricted completion of $\Igfr_0$ (and thus also $\Igfr^\bbf$) at a closed point is
  \[
   W(\FFbar_p\pot{x_1^{1/p^\infty},\ldots,x_{2\langle \rho, \nu \rangle}^{1/p^\infty}}) = \breve\ZZ_p \pot{x_1^{1/p^\infty},\ldots,x_{2\langle \rho, \nu \rangle}^{1/p^\infty}}^{\wedge\, p-adic}
  \]
 \end{proof}

 \begin{proposition} \label{prop CS product}
  \begin{subenv}
   \item There exists a unique lift $\tilde\pi_\infty\colon \Igfr^\bbf \times \Mfr_{G,\mu}(\bbf) \to \Sscr_G$ of $\tilde\pi_\infty'$ whose restriction to the underlying reduced subscheme equals $\pi_\infty$.
   \item Let $\Xfr^\bbf$ be the subfunctor of $\Xfr^{\bbf'}_{/\Sscr_G}$, evaluated on objects in ${\rm Nilp}_{\breve\ZZ_p}^{es}$, given by the additional assumption that the quasi-isogeny $\rho$ respects crystalline Tate tensors. Then the pullback of $\Igfr^{\bbf'} \times \Mfr_{\GSp_{2g},\mu'} \isom \Xfr^{\bbf'}$ restricts to an isomorphism $\Igfr^{\bbf} \times \Mfr_{G,\mu}(\bbf) \isom \Xfr^\bbf$.
  \end{subenv}
 \end{proposition}
 \begin{proof}
  Let $\widetilde\Tfr := (\Igfr^\bbf \times \Mfr_{G,\mu}(\bbf)) \times_{\Ascr_g} \Sscr_G$ and $\tilde\nu\colon \widetilde\Tfr \to \Igfr^\bbf \times \Mfr_{G,\mu}(\bbf)$ be the canonical projection. By Corollary~\ref{cor lift to normalisation} we have to construct a section of $\tilde\nu$. Recall that the proof of Theorem~\ref{thm almost product} constructs a clopen subset 
  \[ 
  Z_{N,m} := s(\Jscr_m^{(p^{-N})} \times \Mscr_{G,\mu}(\bbf)^{m_1,m_2}) \subset (\Jscr_m^{(p^{-N})} \times \Mscr_{G,\mu}(\bbf)^{m_1,m_2}) \times_{\Ascr_g} \Sscr_G
  \]
  which is mapped isomorphically onto $\Jscr_m^{(p^{-N})} \times \Mscr_{G,\mu}(\bbf)^{m_1,m_2}$ by the canonical projection. Denote by $\Zfr$ the union of connected components of $\widetilde\Tfr$ whose underlying reduced subscheme equals the limit of $Z_{N,m}$. In particular $\nu_{|\Zfr}$ induces an isomorphism of the underlying reduced subschemes. Moreover, $\tilde\nu$ induces an isomorphism of restricted formal neighbourhoods at every point $z \in \Zfr(k)$ by Lemma~\ref{lem local product structure} and is thus \'etale at $z$ by Corollary~\ref{cor criterion for etale}. Hence $\tilde\nu$ is an isomorphism on a neighbourhood of $z$. As $\Zfr$ is Jacobson, $\tilde\nu_{|\Zfr}$ is an isomorphism, thus its inverse gives us a section of $\tilde\nu$.
  
 Consider the image $\Xfr_0^{\bbf}$ of $\Igfr^{\bbf} \times \Mfr_{G,\mu}(\bbf)$ in $\Xfr^{\bbf'}_{/\Sscr_G}$. We have $\Xfr^\bbf_0 \subset \Xfr^b$ by construction. More precisely, $\Xfr^\bbf_0$ is a \emph{closed} subfunctor of $\Xfr^\bbf$ since $\Igfr^\bbf \to \Igfr^{\bbf'}$ and $\Mfr_{G,\mu}(\bbf) \to \Mfr_{\GSp_{2g},\mu'}(\bbf')$ are closed immersions by Proposition~\ref{prop comparision of Igusa varieties} and \cite[Prop.~3.11]{HowardPappas:GSpin}.

 In Proposition~\ref{prop infinite almost product}, we have already shown that $\Xfr_0^{\bbf}(\FFbar_p) = \Xfr^\bbf(\FFbar_p)$. Given a point in $\Xfr^\bbf(R)$ with $R \in {\rm Nilp}_{\breve\ZZ_p}^{es}$ arbitrary, let $\Spf R'$ denote the closed formal subscheme which factorises through $\Xfr_0^\bbf$. Then $\Spf R'$ contains all closed points of $\Spf R$ by the previous observation and moreover their restricted formal neighbourhoods by the same argument as in the proof of Lemma~\ref{lem local product structure}. Thus $\Spf R' \mono \Spf R$ is \'etale by Corollary~\ref{cor criterion for etale} and hence an isomorphism. Thus $\Xfr_0^\bbf = \Xfr^\bbf$.

 \end{proof}

 By the same arguments as in \cite{CaraianiScholze:ShVar}, one gets deduces the product structure on the generic fibre . Denote by $A_G,S_G,,Ig^\bbf,M_{G,\mu}(\bbf)$ and $X^\bbf$ the respective adic generic fibres over $\breve E[\zeta_{p^\infty}]$ of $\Acal_G,\Sscr_G, \Igfr^\bbf, \Mfr_{G,\mu}(\bbf)$ and $\Xfr^\bbf$. 
 
 We briefly recall the definition of the infinite level Rapoport Zink space $M_{G,\mu,\infty}$ (for details, see \cite[\S~7.6]{Kim:RZ}). Denote by $(\Xcal_{\rm RZ},\lambda_{\rm RZ},(t_{\rm RZ, \alpha}))$ the universal object over $\Mfr_{G,\mu}(\bbf)$  and we denote by $T(\Xcal_{\rm RZ}) \coloneqq \varprojlim \Xcal_{\rm RZ}[p^n]$ its Tate module. We define $M_{G,\mu,\infty}^\diamond(\bbf)$ as the (preperfectoid) adic space parametrising morphisms $\alpha\colon M \to (T_p\Xcal_{\rm RZ})^{\rm ad}$ (where the free $\ZZ_p$-module $M$ is considered as constant adic space), such that the pairing $\psi$ on $M$ matches the pairing on $(T_p\Xcal_{\rm RZ})^{\rm ad}$ induced by $\lambda_{\rm RZ}$ and such that $\alpha$ becomes an isomorphism whenever it is restricted to some geometric point (cf.~\cite[Thm.~4.2.4]{CaraianiScholze:ShVar}). On the other hand, the adic spaces $ (\Xcal_{\rm RZ}[p^n])^{\rm ad}$ form a system of finite \'etale coverings of $M_{G,\mu}$ which corresponds to a lisse $\ZZ_p$-sheaf $T_p(X_{\rm RZ}^{\rm ad})$ on $M_{G,\mu}$. Moreover the the crystalline Tate-tensors $(t_{\rm RZ, \alpha})$ induce tensors on $T_p(X_{\rm RZ}^{\rm ad})$ (\cite[Thm.~7.1.6]{Kim:RZ}). Then $\alpha$ induces an isomorphism $\alpha_{et}$ between $M$ (now considered as constant $\ZZ_p$-sheaf) and the pullback of $T_p(X_{\rm RZ}^{\rm ad})$ to $M_{G,\mu,\infty}^\diamond$. We denote by $M_{G,\mu,\infty} \subset M_{G,\mu,\infty}^\diamond$ the locus where $\alpha_{et}$ is compatible with tensors on both sides. This is a closed adic subspace, in particular $M_{G,\mu,\infty}$ is again preperfectoid (see also \cite[Prop.~7.6.1]{Kim:RZ}). We define the adic space $X_{\infty}^\bbf$ over $X^\bbf$ by an analogous construction.

 \begin{corollary}
  We have an isomorphism
  \[
   Ig^\bbf \times M_{G,\mu,\infty}(\bbf) \isom X_{\infty}^\bbf.
  \]
  In particular, $X_{\infty}^\bbf$ is preperfectoid.
 \end{corollary}
 \begin{proof}
  This is a direct consequence of the isomorphisms
  \begin{align*}
   Ig^\bbf \times M_{G,\mu}(\bbf) & \isom X^\bbf \\
   X_G^\bbf \times_{M_{G,\mu}(\bbf)} M_{G,\mu,\infty}(\bbf) & \isom X_{\infty}^\bbf,
  \end{align*}
  where the first isomorphism follows from Proposition~\ref{prop CS product} and the second from the moduli description.
 \end{proof}
 
 \def\cprime{$'$}
\providecommand{\bysame}{\leavevmode\hbox to3em{\hrulefill}\thinspace}
\providecommand{\MR}{\relax\ifhmode\unskip\space\fi MR }
\providecommand{\MRhref}[2]{%
  \href{http://www.ams.org/mathscinet-getitem?mr=#1}{#2}
}
\providecommand{\href}[2]{#2}

 \end{document}